\documentclass[11pt,oneside]{amsart}
\usepackage[top=25mm, bottom=25mm, left=25mm, right=25mm]{geometry}
\usepackage{amsfonts,amssymb,amsmath,amsthm,amsxtra}
\usepackage[T1]{fontenc}
\usepackage[utf8]{inputenc}
\usepackage{bm}
\usepackage{mathtools}
\usepackage{dsfont}
\usepackage{bbm}
\usepackage{mathrsfs}
\usepackage{enumitem}

\usepackage{graphics}

\usepackage{color}

\numberwithin{equation}{section}
\newtheorem{theorem}[equation]{Theorem}

\newtheorem{lemma}[equation]{Lemma}
\newtheorem{proposition}[equation]{Proposition}

\theoremstyle{definition}
\newtheorem{remark}[equation]{Remark}

\newcommand{\CC}{\mathbb{C}}
\newcommand{\DD}{\mathbb{D}}

\newcommand{\GG}{\mathbb{G}}

\newcommand{\II}{\mathbb{I}}
\newcommand{\JJ}{\mathbb{J}}

\newcommand{\NN}{\mathbb{N}}

\newcommand{\QQ}{\mathbb{Q}}
\newcommand{\RR}{\mathbb{R}}

\newcommand{\TT}{\mathbb{T}}
\newcommand{\UU}{\mathbb{U}}

\newcommand{\ZZ}{\mathbb{Z}}

\newcommand{\N}{\mathbb{N}}

\newcommand{\calP}{\mathcal{P}}
\newcommand{\calF}{\mathcal{F}}
\newcommand{\calB}{\mathcal{B}}
\newcommand{\calM}{\mathcal{M}}

\newcommand{\calS}{\mathcal{S}}

\newcommand{\ex}{\bm{e}}

\newcommand*{\DMO}[1]{\expandafter\DeclareMathOperator\csname #1\endcsname {#1}}
\DMO{ord}
\DMO{supp}
\DMO{Sym}
\DMO{cl}
\DMO{lcm}
\DMO{dist}
\DMO{tr}
\DMO{diam}

\DeclarePairedDelimiter\abs{\lvert}{\rvert}

\DeclarePairedDelimiter\norm{\lVert}{\rVert}
\DeclarePairedDelimiter\card{\lvert}{\rvert}
\DeclarePairedDelimiter\floor{\lfloor}{\rfloor}
\DeclarePairedDelimiter\ceil{\lceil}{\rceil}
\DeclarePairedDelimiterX\spr[2]{\langle}{\rangle}{#1,#2}
\newcommand{\ipr}[2]{#1\cdot#2}
\DeclarePairedDelimiterX\Set[2]{\{}{\}}{#1\colon #2}
\DeclarePairedDelimiterX\Seq[1]{(}{)}{#1}
\newcommand\blfootnote[1]{%
  \begingroup
  \renewcommand\thefootnote{}\footnote{#1}%
  \addtocounter{footnote}{-1}%
  \endgroup
}
\begin{document}
\title[Bootstrap methods in bounding discrete Radon operators]
	{Bootstrap methods in bounding discrete Radon operators}

\author{Wojciech S{\l}omian}
\address[Wojciech S{\l}omian]{
	  Faculty of Pure and Applied Mathematics, 
	  Wroc\l{}aw University of Science and Technology\\
	  Wyb{.} Wyspia\'nskiego 27,
	  50-370 Wroc\l{}aw, Poland}
    \email{wojciech.slomian@pwr.edu.pl}
\thanks{The author was
partially supported by the National Science Centre in Poland, grant
Opus 2018/31/B/ST1/00204.}

\begin{abstract}
The aim of this paper is to  develop  bootstrap arguments to establish maximal, oscillation, variational and jump inequalities for the discrete averaging Radon operators on $\ell^p(\mathbb Z^d)$. 
\end{abstract}

\maketitle
		\blfootnote{
The paper is a part of author's doctoral thesis written under the supervision of Professor Mariusz Mirek.
}
\section{Introduction}
\label{sec:1}
Various kinds of seminorms have been intensively studied since Bourgain's seminal papers \cite{B1,B2,B3}, where oscillation and variation seminorms were used to prove the pointwise convergence of the ergodic averages along polynomials. As many authors generalised and extended Bourgain's work, oscillation, variational and jump seminorms became widely used tools in harmonic analysis and ergodic theory. This paper is motivated by the work of Mirek, Stein and Zorin-Kranich \cite{MSZ2} in which they proved, among others, that the so-called bootstrap approach can be used to prove variational and  jump inequalities for continuous Radon operators. Our purpose is to show that one can deal with maximal, oscillation, variational and jump inequalities in the context of discrete Radon operators by using an analogous technique. This approach allows us to prove seminorms inequalities without referring to vector-valued or square function estimates as it was in previously known proofs.  By using this technique we give new proofs of well-known variational and jump results from \cite{MST2} and \cite{MSZ3}, as well as a new proof of the oscillation inequality, recently proven in \cite{MSS}.\\
\indent Let us consider a polynomial mapping
\begin{equation}\label{polymap}
    \mathcal{P}=(\mathcal{P}_1,\dots,\mathcal{P}_{d})\colon\ZZ^k\to\ZZ^{d},
\end{equation}
where each $\calP_j\colon\ZZ^k\to\ZZ$ is a polynomial of $k$ variables with integer coefficients such that $\calP_j(0)=0$. Let $\Omega$ be a bounded convex open subset of $\RR^k$. For any $t>0$ we set
\[
\Omega_t := \{x\in\RR^{k}: t^{-1}x\in\Omega\}.
\]
Moreover, we assume that $B(0,c_{\Omega}) \subseteq \Omega \subseteq B(0,1)$ for some $c_{\Omega}\in(0, 1)$, where $B(x, t)$ is the open Euclidean ball in $\RR^k$ with radius $t>0$ centered at a point $x\in\RR^k$. Here one can think that $\Omega_{t}$ is the open ball of radius $t$ associated with some norm on $\RR^{k}$.

For finitely supported functions $f\colon\ZZ^{d}\to \CC$ and $t>0$, we define the discrete averaging Radon operator by setting
\begin{align}
\label{eq:1}
M_t^\calP f(x):= \frac{1}{\card{\Omega_{t}\cap\ZZ^{k}}} \sum_{y \in\Omega_{t}\cap\ZZ^{k}} f(x-\calP(y)),\quad x\in\ZZ^d.
\end{align}
\indent In the case of seminorms we will follow the notation from \cite{MSS} and \cite{MSZ3}. Let $\II\subseteq\RR$. For a given increasing sequence $I=(I_j: j\in\NN)\subseteq\II$, the {\it truncated oscillation seminorm} is defined for any function $f\colon \II \to \CC$ by setting
\begin{align}
\label{eq:45}
O^2_{I,N}(f):=O_{I, N}^2( f(t): t\in \II)
:= \Big(\sum_{j=1}^N\sup_{\substack{I_j \le t < I_{j+1}\\t\in\II}}
\abs{f(t)-f(I_j)}^2\Big)^{1/2}
\quad \text{for all} \quad N\in\NN\cup\{\infty\}.
\end{align}
A closely related concept to the oscillation seminorm is the {\it $r$-variation seminorm}. Recall that for any $r\geq0$ the $r$-variation seminorm $V^r$ of a function $f\colon\II \to \CC$ is defined by
\begin{align}
\label{eq:def:Vr}
V^{r}(f)
:=V^{r}(f(t) : t\in\II):=
\sup_{J\in\N} \sup_{\substack{t_{0}<\cdots<t_{J}\\ t_{j}\in\II}}\Big(\sum_{j=0}^{J-1}  \abs{f(t_{j+1})-f(t_{j})}^{r} \Big)^{1/r}, 
\end{align}
where the latter supremum is taken over all finite increasing sequences in $\II$. In the case of $r=\infty$ we have an obvious modification. Finally, for any $\lambda>0$ and $\II \subseteq \RR$, the \emph{$\lambda$-jump counting function} of a function $f\colon\II \to \CC$ is defined by
\begin{align}
  \label{eq:def:jump}
N_{\lambda}(f) :=
N_{\lambda}(f(t) : t\in\II):=\sup \{J\in\NN\, |\, \exists_{\substack{t_{0}<\dotsb<t_{J}\\ t_{j}\in\II}}  : \min_{0<j\leq J} \abs{f(t_{j})-f(t_{j-1})} \geq \lambda\}.
\end{align}
The aim of this article is to give new proofs of seminorm inequalities from \cite{MSS}, \cite{MST2} and \cite{MSZ3} by using the discrete Littlewood--Paley theory (see \cite[Theorem 3.3]{M10}) and the bootstrap arguments similar in spirit to these from \cite{DR86}. This approach is much in the spirit of the methods used in the continuous setting in which we usually use Littlewood--Paley operators, see \cite{MSZ2}. It is also more ``elementary'' since it does not require any external knowledge like vector-valued estimates from \cite{MST1} or the estimates for appropriate square functions like in \cite{MSZ3}. Our main result is to give a bootstrapping proof of the following theorem.
\begin{theorem}\label{Thm1:bootstrap}
Let $d, k\ge 1$ and let $\calP$ be a polynomial mapping~\eqref{polymap}. Then for any $p\in(1,\infty)$ there is a constant $C_{p,d,k,{\rm deg}\calP}>0$ such that for any $f\in \ell^p(\ZZ^d)$ we have
\begin{align}
    \norm[\big]{\sup_{t>0}|M_t^\calP f|}_{\ell^p(\ZZ^d)}&\le C_{p,d,k,{\rm deg}\calP}\|f\|_{\ell^p(\ZZ^d)},\label{maximalest}\\
    \sup_{\lambda>0}\norm[\big]{\lambda N_{\lambda}(M_t^\calP f:t\in\RR_+)^{1/2}}_{\ell^p(\ZZ^d)}&\le C_{p,d,k,{\rm deg}\calP}\|f\|_{\ell^p(\ZZ^d)},\label{jumpest}\\
    \sup_{N\in\NN}\sup_{I\in\mathfrak{S}_N(\RR_+)}\norm[\big]{O_{I,N}^2(M_t^\calP f:t\in\RR_+)}_{\ell^p(\ZZ^d)}&\le C_{p,d,k,{\rm deg}\calP}\norm{f}_{\ell^p(\ZZ^d)},\label{oscillest}
\end{align}
where $\mathfrak{S}_N(\RR_+)$ is the set of all strictly increasing sequences of length $N+1$ contained in $\RR_+$ (see Section~\ref{sec:notation}). Moreover, for any $r\in(2,\infty)$ there is a constant $C_{p,d,k,r,{\rm deg}\calP }>0$ such that
\begin{equation}
    \norm[\big]{V^r(M_t^\calP f:t\in\RR_+)}_{\ell^p(\ZZ^d)}\le C_{p,d,k,r,{\rm deg}\calP }\|f\|_{\ell^p(\ZZ^d)},\quad f\in\ell^p(\ZZ^d).\label{varest}
\end{equation}
In particular, the constants mentioned above are independent of the
coefficients of the polynomial mapping $\calP$.
\end{theorem}
A few comments are in order. The maximal inequality \eqref{maximalest} was proven in \cite{MT1} and it is a generalisation of Bourgain's result from \cite{B3}. The jump inequality \eqref{jumpest} was proven in \cite{MSZ3}. The oscillation inequality \eqref{oscillest} was recently proven in \cite{MSS} and the first proof of the $r$-variation estimate \eqref{varest} in a full range $r>2$ was given in \cite{MST2} (see also \cite{zk} for previous results). It is worth to note that the inequality \eqref{jumpest} implies $r$-variation estimates for $r\in(2,\infty)$, and each of these, in turn, implies the maximal inequality \eqref{maximalest} -- for details see Section~\ref{sec:seminorms}. Only the oscillation inequality \eqref{oscillest} is not implied by the former ones. Hence, one could just prove inequalities \eqref{jumpest} and \eqref{oscillest} and conclude that the other inequalities are implied by them. However, in this paper we would like to show that the Littlewood--Paley theory and the bootstrap approach can be used in the case of all mentioned seminorms. Fortunately, it does not greatly affect the length of the article.
\subsection{Historical background}
\indent The oscillation seminorm \eqref{eq:45} was introduced in Bourgain's articles \cite{B1,B2,B3} to prove that for any dynamical system $(X,\calB,\mu,T)$ with a measure preserving transformation $T$ the averages along the squares
\begin{equation*}
    T_N f(x):=\frac{1}{N}\sum_{n=1}^N{f(T^{n^2}x)}
\end{equation*}
converge $\mu$-almost everywhere as $N$ goes to infinity for all $f\in L^p(X,\mu)$ with $p\in(1,\infty)$. The usual approach to prove pointwise convergence requires $L^p$ boundedness for the corresponding maximal function and pointwise convergence on a dense class of functions in $L^p(X,\mu)$. In the case of the operator $T_N$ it is not easy to find an appropriate dense class for which the pointwise convergence is {\it a priori} known. Fortunately, Bourgain managed to bypass that problem by using the oscillation seminorm. He proved that for a lacunary sequence $I=(I_j: j\in\NN)$ and 
for any $J\in\NN$ one has
\begin{equation}\label{eq:int:1}
    \norm[\big]{O_{I,J}^2(T_N f:N\in\NN)}_{L^2(X)}\lesssim_I J^{c}\norm{f}_{L^2(X)}
\end{equation}
for some a constant $c<1/2$. From the above inequality one may deduce that $T_N f$ converges pointwise for $f\in L^2(X,\mu)$. Later, many authors followed Bourgain's approach and studied oscillation inequalities in various contexts, see \cite{CJRW1,jkrw,MSS}.\\
\indent To prove \eqref{eq:int:1} Bourgain used $r$-variation seminorms $V^r$, a well known object from the martingale theory. The $r$-variations for a family of bounded martingales
$(\mathfrak f_n\colon X\to \CC:n\in\NN)$ were studied by L{\'e}pingle
\cite{Lep} who showed that for all $r\in(2, \infty)$ and
$p\in(1, \infty)$ there is a constant $C_{p, r}>0$ such that the following inequality
\begin{align}
\label{eq:150}
\norm{V^r(\mathfrak f_n: n\in\NN)}_{L^p(X)}
\le C_{p, r}\sup_{n\in\NN}\norm{\mathfrak f_n}_{L^p(X)}
\end{align}
holds with sharp ranges of exponents, see also \cite{JG} for a counterexample at $r=2$. We also refer to \cite{B3,MSZ1,PX} for generalisations and different proofs of \eqref{eq:150}. The inequality \eqref{eq:150} is an extension of Doob's maximal inequality for martingales and is a quantitative form of the martingale convergence theorem. The $r$-variation is closely related to the oscillation seminorm since by H\"older's inequality, for $r\ge2$, we have 
\begin{equation}\label{eq:int:2}
    O_{I,J}^2(T_N f:N\in\NN)\le J^{1/2-1/r}V^r(T_N f:N\in\NN).
\end{equation}
Bourgain approximated the operator $T_N$ by a suitable martingale and then used L{\'e}pingle's variational inequality to show that for $r>2$
\begin{equation*}
    \|V^r(T_N f:N\in\NN)\|_{L^p(X)}\lesssim\|f\|_{L^p(X)}.
\end{equation*}
Having this, one easily obtains \eqref{eq:int:1} using the inequality \eqref{eq:int:2}.

In \cite{B3} Bourgain also considered the $\lambda$-jump counting function which turns out to be even more fundamental object than the $r$-variation $V^r$. It is not hard to obtain that for every $r\ge1$ one has 
\begin{align}
\label{eq:17}
\sup_{\lambda>0}\|\lambda N_{\lambda}(T_N f:N\in\NN)^{1/r}\|_{L^p(X)}\le \|V^r(T_N f:N\in\NN)\|_{L^p(X)}.
\end{align}
The remarkable feature of the $\lambda$-jumps, observed by Bourgain \cite{B3}, is that, in some sense, the inequality \eqref{eq:17} can be reversed. Namely, {\it a priori} uniform $\lambda$-jump estimates 
\begin{align}
\label{eq:18}
\sup_{\lambda>0}\|\lambda N_{\lambda}(T_N f:N\in\NN)^{1/2}\|_{L^p(X)}
\le C_{p}\|f\|_{L^p(X)},
\end{align}
for some $p\in[1, \infty)$ imply weak $r$-variational estimates 
\begin{align}
\label{est2}
\|V^r(T_N f:N\in\NN) \|_{L^{p,\infty}(X)}
\le C_{p,r}\|f\|_{L^p(X)},
\end{align}
for the same value of $p$ and for all $r\in(2, \infty]$.
For more details we refer to \cite{jsw} and \cite{MSZ1, MSZ3}, and the references given there.

Bourgain's papers \cite{B1, B2, B3} were the starting point of comprehensive investigations  in  ergodic theory and harmonic analysis, which resulted in many papers, see for instance \cite{AJS,JG,jsw,MST2,MSZ1,MSZ2,MSZ3,MT1}.\\
\indent The discrete Radon averages $M_t^\calP$ arise upon applying Calderón's transference principle \cite{Cald} to the ergodic averages. Let $(X,\mathcal{B},\mu)$ be a $\sigma$-finite measure space with a family of invertible, commuting and measure preserving transformations $T_1, T_2,\ldots, T_d$. For a given polynomial mapping of the form \eqref{polymap} we define the ergodic averages
\begin{equation*}
    A_t^\calP f(x):=\frac{1}{|\Omega_t\cap\ZZ^k|}\sum_{y\in\Omega_t\cap\ZZ^k}f(T_1^{\calP_1(y)}T_2^{\calP_1(y)}\cdots T_d^{\calP_d(y)}x).
\end{equation*}
If one wants to prove, for example, the maximal inequality
\begin{equation*}
    \norm[\big]{\sup_{t>0}|A_t^\calP f|}_{L^p(X)}\lesssim\|f\|_{L^p(X)},
\end{equation*}
then by Calderón's transference principle it is enough to prove the corresponding maximal inequality for appropriate Radon averages $M_t^\calP$ on $\ZZ^d$. 
Nonetheless, the above observation is not the only reason to consider Radon averages. Namely, the operators $M_t^\calP$ can be seen as discrete counterparts of the continuous Radon operators
\begin{equation*}
    \calM_t^\calP f(x):=\frac{1}{|\Omega_t|}\int_{\Omega_t}f(x-\calP(y)){\rm d}y
\end{equation*}
being, in turn, a natural generalisation of the Hardy--Littlewood operators. This direction originates in some problems related to curvatures, see \cite{CNSW,IMSW,SW3}. It is well known (see for instance \cite[Chapter 11]{bigs}) that the maximal operator $\sup_{t>0}|\calM_t^\calP|$ is bounded on $L^p(\RR^d)$ with $p>1$. In 2008, Jones, Seeger and Wright \cite{jsw} managed to prove the jump inequality for the operators $\calM_t^\calP$ over Euclidean balls, see also the appendix in \cite{MST2}. For operators associated with arbitrary convex bodies the jump inequality was proven by Mirek, Stein and Zorin-Kranich \cite{MSZ2} by using bootstrap arguments. Finally, the oscillation inequality was recently proven in \cite{MSS} by the author in collaboration with Mirek and Szarek.\\
\indent The idea of the bootstrap approach originates in the work of Nagel, Stein and Wainger \cite{NSW} where they studied the problem of differentiation in lacunary directions. Their argument relied heavily on some geometrical considerations. Later, Duoandikoetxea and Rubio de Francia \cite{DR86} used the ideas of bootstrap from  \cite{NSW} (see Lemma~\ref{lemma28} below)  to prove $L^p$ bounds for maximal Radon transform. At the same time, Christ formulated the bootstrap argument from \cite{NSW} in a fairly abstract way, which was used and published by 
Carbery \cite{Car1, Car2}. Finally, Mirek, Stein and Zorin-Kranich \cite{MSZ2} managed to use the bootstrap argument to establish jump inequalities in a very abstract setting.

\indent In order to prove Theorem~\ref{Thm1:bootstrap} we exploit a procedure introduced in \cite{MST2} in the context of $r$-variations. The key ingredient is the vector-valued Ionescu--Wainger multiplier theorem from \cite{MSZ3} and the discrete Littlewood--Paley theory which follows from it (see \cite{M10}). Finally, we use the bootstrap lemma from \cite{MSZ2} which allows us to bound the appropriate square function by the power of the appropriate seminorm.
\section{Preliminaries}
\subsection{Basic notation}\label{sec:notation}
We denote $\NN:=\{1, 2, \ldots\}$, $\NN_0:=\{0,1,2,\ldots\}$ and $\RR_+:=(0, \infty)$. For
$d\in\NN$ the sets $\ZZ^d$, $\RR^d$, $\CC^d$ and $\TT^d\equiv[-1/2, 1/2)^d$ have the usual meaning. For every $N\in\NN$ we  define
    \[
    \NN_N:=\{1,\ldots, N\}.
    \]  
For any $x\in\RR$ the floor function is defined by
\[
\lfloor x \rfloor: = \max\{ n \in \ZZ : n \le x \}.
\]
We write $A \lesssim B$ to indicate that $A\le CB$ with a constant $C>0$. The constant $C$ may change from line to line. We write
$\lesssim_{\delta}$ if the implicit constant depends on $\delta$. For two functions $f\colon X\to \CC$ and
 $g\colon X\to [0, \infty)$, we write $f = \mathcal{O}(g)$ if
there is a constant $C>0$ such that $|f(x)| \le C g(x)$ for all
$x\in X$.

For $N\in\NN\cup\{\infty\}$ we denote by $\mathfrak S_N(\II)$ the family of all strictly increasing
sequences of length $N+1$ contained in $\II$.

\subsection{Norms and inner product} 
In the paper we use the standard inner product on $\RR^d$ denoted by $ x\cdot\xi$. Moreover, for any $x\in\RR^d$ we use the the $\ell^2$-norm and the maximum norm defined by
\begin{align*}
\abs{x}:=\abs{x}_2:=\sqrt{\ipr{x}{x}}, \qquad \text{ and } \qquad |x|_{\infty}:=\max_{1\leq k\leq d}|x_k|.
\end{align*}

For any multi-index $\gamma=(\gamma_1,\dots,\gamma_k)\in\N^k_0$, by
abuse of notation we will write
$|\gamma|:=\gamma_1+\cdots+\gamma_k$. This should never cause any confusions
since the multi-indices are always denoted by Greek letter $\gamma$.
\subsection{Fourier transform}
Let $\GG=\RR^d$ or $\GG=\ZZ^d$ and let $\GG^*$ denote the dual group of $\GG$. For every $z\in\CC$ we set $\ex(z):=e^{2\pi {\bm i} z}$, where ${\bm i}^2=-1$. Let $\calF_{\GG}$ denote the Fourier transform on $\GG$ defined for any $f \in L^1(\GG)$ as
\begin{align*}
\calF_{\GG} f(\xi) := \int_{\GG} f(x) \ex(x\cdot\xi) {\rm d}\mu(x),\quad \xi\in\GG^*,
\end{align*}
where $\mu$ is the Haar measure on $\GG$. For any bounded function $\mathfrak m\colon \GG^*\to\CC$ we define the Fourier multiplier operator by 
\begin{align}
\label{eq:2}
T_{\GG}[\mathfrak m]f(x):=\int_{\GG^*}\ex(-\xi\cdot x)\mathfrak m(\xi)\calF_{\GG}f(\xi){\rm d}\xi, \quad \text{ for } \quad x\in\GG.
\end{align}
Here, we assume that $f\colon\GG\to\CC$ is a compactly supported function on $\GG$ (and smooth if $\GG=\RR^d$) or any other function for which \eqref{eq:2} makes sense.
\subsection{Ionescu--Wainger multiplier theorem}\label{sec:IW}
The important tool in dealing with discrete singular integral will be the Ionescu--Wainger multiplier theorem. The following result is the multidimensional vector-valued Ionescu--Wainger multiplier theorem from \cite[Section 2]{MSZ3}.
\begin{theorem}\label{thm:IW-mult}
For every $\varrho>0$, there exists a family $(P_{\leq N})_{N\in\NN}$ of subsets of $\NN$ such that:  
\begin{enumerate}[label*={(\roman*)}]
\item \label{IW1} One has $\NN_N\subseteq P_{\leq N}\subseteq\N_{\max\{N, e^{N^{\varrho}}\}}$.
\item \label{IW2}  If $N_1\le N_2$, then $P_{\leq N_1}\subseteq P_{\leq N_2}$.
\item \label{IW3}  If $q \in P_{\leq N}$, then all factors of $q$ also lie in $P_{\leq N}$.
\item \label{IW4}  One has $\lcm{(P_N)}\le 3^N$.
\end{enumerate}

Furthermore, for every $p \in (1,\infty)$, there exists  $0<C_{p, \varrho, d}<\infty$ such that, for every $N\in\NN$, the following holds.

Let $0<\varepsilon_N \le e^{-N^{2\varrho}}$, and let $\Theta\colon\RR^{d} \to L(H_0,H_1)$ be a measurable function supported on $\varepsilon_{N}\mathbf Q$, where $\mathbf Q:=[-1/2, 1/2)^d$ is a unit cube, with values in the space $L(H_{0},H_{1})$ of bounded linear operators between separable Hilbert spaces $H_{0}$ and $H_{1}$.
Let $0 \leq \mathbf A_{p} \leq \infty$ denote the smallest constant such that, for every function $f\in L^2(\RR^d;H_0)\cap L^{p}(\RR^d;H_0)$, we have
\begin{align}
\label{eq:75}
\norm[\big]{T_{\RR^d}[\Theta]f}_{L^{p}(\RR^{d};H_1)}
\leq
\mathbf A_{p} \norm{f}_{L^{p}(\RR^{d};H_0)}.
\end{align}
Then the multiplier
\begin{equation}
\label{eq:IW-mult}
\Delta_N(\xi)
:=\sum_{b \in\Sigma_{\leq N}}
\Theta(\xi - b),
\end{equation}
where $\Sigma_{\leq N}$ is $1$-periodic subsets of $\TT^d$ defined by
\begin{align}
\label{eq:42}
\Sigma_{\leq N} := \Big\{ \frac{a}{q}\in\QQ^d\cap\TT^d:  q \in P_{\leq N} \text{ and } {\rm gcd}(a, q)=1\Big\},
\end{align}
satisfies for every $f\in L^p(\ZZ^d;H_0)$ the following inequality
\begin{align}
\label{eq:76}
\norm[\big]{ T_{\ZZ^d}[\Delta_{N}]f}_{\ell^p(\ZZ^{d};H_1)}
\le C_{p,\varrho,d}
(\log N) \mathbf A_{p}
\norm{f}_{\ell^p(\ZZ^{d};H_0)}.
\end{align}
\end{theorem}

An important feature of  Theorem~\ref{thm:IW-mult} is that one can
directly transfer square function estimates from the continuous to the
discrete setting. By property \ref{IW1} we see that
\begin{align}
\label{eq:43}
|\Sigma_{\leq N}|\lesssim e^{(d+1)N^\varrho}.
\end{align}
 A detailed proof of Theorem~\ref{thm:IW-mult} (in the spirit of
\cite{M10}) can be found in \cite[Section 2]{MSZ3}. It turns out that factor $\log N$ may be
removed from \eqref{eq:76}, see \cite{TaoIW}.
\subsection{Lifting procedure}\label{sec:lift}
By a standard lifting argument it suffices to
prove Theorem~\ref{Thm1:bootstrap} only for canonical polynomial mappings.
Let $\mathcal{P}$ be a polynomial mapping (\ref{polymap}). We define 
\begin{equation*}
    {\rm deg}\, \mathcal{P}:=\max\{{\rm deg}\, \mathcal{P}_j: 1\le j\le d\}.
\end{equation*}
Let us consider the set of multi-indexes 
\begin{equation*}
    \Gamma:=\big\{\gamma\in\N_0^k\setminus\{0\}: 0<|\gamma|\le {\rm deg}\, \mathcal{P}\big\}
\end{equation*}
equipped with the lexicographic order. By $\RR^\Gamma$ we mean the space of tuples of real numbers labeled by multi-indices $\gamma=(\gamma_1,\dots,\gamma_k)$. It can be easily noted that $\RR^\Gamma\cong\RR^{|\Gamma|}$ and $\ZZ^\Gamma\cong\ZZ^{|\Gamma|}$. Recall that the canonical polynomial mapping is defined by
\begin{equation*}
    \RR^k\ni x=(x_1,\dots,x_k)\mapsto(x)^\Gamma:=(x^\gamma\colon\gamma\in\Gamma)\in\RR^\Gamma,
\end{equation*} 
where $x^\gamma=x_1^{\gamma_1}x_2^{\gamma_2}\cdots x_k^{\gamma_k}$. In the case of the averaging Radon operator associated with with the canonical mapping we will write
\begin{equation*}
    M_t f(x):= \frac{1}{|\Omega_t\cap\ZZ^k|}\sum_{y \in\Omega_{t}\cap\ZZ^{k}} f(x-(y)^\Gamma).
\end{equation*}
By invoking the lifting procedure for the Radon averages described in \cite[Lemma 2.2]{MST1} (see also \cite[Section 11]{bigs}) it is enough to prove Theorem~\ref{Thm1:bootstrap} only for the canonical mappings.
\subsection{Properties of the considered seminorms}\label{sec:seminorms}
The proof of Theorem~\ref{Thm1:bootstrap} will be conducted simultaneously for all considered (quasi-)seminorms, pointing out the places where the distinction should be made. Due to  technical reasons we will consider the following seminorm 
\begin{equation}\label{supnew}
    \sup_{t>0}|M_t f- M_0f|
\end{equation}
instead of the usual maximal function. We could prove Theorem~\ref{Thm1:bootstrap} for the usual maximal function but it would require a few minor changes which would affect the length of the article. The seminorm \eqref{supnew} is more consistent with the rest of the seminorms and allows us to write results in a more elegant way.

Now let $\II\subseteq\RR$. For a given family of measurable
functions $(\mathfrak{a}_t:t\in\II)\subset\ell^p(\ZZ^\Gamma)$ we consider the following seminorms: the supremum seminorm
\begin{equation}
     \norm[\big]{\sup_{t\in\II}|\mathfrak a_t-\mathfrak a_{\inf \II}|}_{\ell^p(\ZZ^\Gamma)},\label{supnorm}
\end{equation}
the oscillation seminorm
\begin{equation}
    \sup_{N\in\NN}\sup_{I\in\mathfrak S_N(\II)}\norm[\big]{O_{I,N}^2(\mathfrak a_t:t\in\II)}_{\ell^p(\ZZ^\Gamma)},\label{oscillnorm}
\end{equation}
the jump quasi-seminorm 
\begin{equation}
     \sup_{\lambda}\norm[\big]{\lambda N_\lambda(\mathfrak a_t:t\in\II)^{1/2}}_{\ell^p(\ZZ^\Gamma)}\label{jumpnorm},
\end{equation}
and, for $r\in[2,\infty)$, the $r$-variation seminorm
\begin{equation}
    \norm[\big]{V^r(\mathfrak{a}_t:t\in\II)}_{\ell^p(\ZZ^\Gamma)}\label{varnorm}.
\end{equation}
 For clarity, for all of the above seminorms we will use one common symbol $$\calS_p(\mathfrak{a}_t:t\in\II).$$

\indent Below we state some basic proprieties of the seminorms \eqref{supnorm}--\eqref{varnorm}. For any $r\in[2,\infty]$ we have the following pointwise estimate for the maximal function
\begin{equation}\label{domsup1}
    \sup_{t\in\II}|\mathfrak{a}_t|\leq V^r(\mathfrak{a}_t)+|\mathfrak{a}_{t_0}|.
\end{equation}
A similar result holds in the case of the oscillation seminorm
\begin{equation}\label{domsup2}
    \norm{\sup_{t\in\II\setminus\sup\II}|\mathfrak{a}_t|}_{\ell^p(\ZZ^\Gamma)}\leq\sup_{t\in\II}\norm{\mathfrak{a}_t}_{\ell^p(\ZZ^\Gamma)}+\sup_{N\in\NN}\sup_{I\in\mathfrak{S}_N(\II)}\norm[\big]{O_{I,N}^2(\mathfrak{a}_t:t\in\II)}_{\ell^p(\ZZ^\Gamma)}.
\end{equation}
The jump quasi-seminorm \eqref{jumpnorm}, in turn, provides the weak type estimate for the $r$-variation seminorm (see \cite[Lemma 2.5]{MSZ1}), that is, for any $r>2$ one has
\begin{equation}\label{domweak}
\big\| V^{r}\big( \mathfrak{a}_t : t \in \II \big) \big\|_{\ell^{p,\infty} (\ZZ^\Gamma)}
\lesssim_{p,r}
\sup_{\lambda>0} \big\| \lambda N_{\lambda}(\mathfrak{a}_t: t\in \II)^{1/2} \big\|_{\ell^{p}(\ZZ^\Gamma)}.
\end{equation}
Let us note that for any countable family $(a_n:n\in\JJ)\subset\ell^p(\ZZ^\Gamma)$ the following inequality holds
\begin{equation}\label{square}
    \calS_p(a_n:n\in \JJ)\le2\norm[\big]{\big(\sum_{n\in \JJ}|\mathfrak{a}_n|^2\big)^{1/2}}_{\ell^p(\ZZ^\Gamma)}.
\end{equation}
\begin{proposition}
Let $p\in(1,\infty)$ and $\II\subseteq\RR$. Then the seminorm $\calS_p$ is subadditive up to some positive constant, that is,
\begin{equation*}
    \calS_p(\mathfrak{a}_t+\mathfrak{b}_t:t\in\II)\lesssim  \calS_p(\mathfrak{a}_t:t\in\II)+ \calS_p(\mathfrak{b}_t:t\in\II),
\end{equation*}
where the implied constant is independent of the set $\II$, and the families $(\mathfrak{a}_t:t\in\II)$ and $(\mathfrak{b}_t:t\in\II)$. Moreover, for $-\infty\le u<w<v\le\infty$, we have
\begin{align}
    \calS_p(\mathfrak{a}_t:t\in [u,v])\lesssim \calS_p(\mathfrak{a}_t\in[u,w+1])+\calS_p(\mathfrak{a}_t:t\in[w,v]),\label{eq:split1}\\
    \calS_p(\mathfrak{a}_t\mathds{1}_{(u,\infty)}(t):t\in[0,w])\lesssim\calS_p(\mathfrak{a}_t:t\in[u,w])+\|\mathfrak a_u\|_{\ell^p(\ZZ^\Gamma)},\label{eq:split2}
\end{align}
where implied constants depend only on the seminorm $\calS_p$.
\end{proposition}
\begin{proof}
The only difficult part of the proof may be subadditivity of the jump quasi-seminorm. However, by \cite[Corollary 2.2]{MSZ1} we know that it admits an equivalent subadditive seminorm which yields the desired result.
\end{proof}
The next result is a well-known decomposition into the long ``jumps'' and the short variations from \cite{jsw} (see also \cite[Lemma 8.1]{MST2}).
\begin{proposition}[{\cite[Lemma 1.3]{jsw}}]\label{prop:longandshort}
Let $p\in[1,\infty)$ and  $\II\subseteq\RR_+$. Then we have the following decomposition into the long ``jumps'' and short variations
\begin{equation}
    \calS_p(\mathfrak a_t:t\in\II)\lesssim \calS_p(\mathfrak a_{2^n}:n\in\ZZ)+\norm[\Big]{\Big(\sum_{n\in\ZZ}V^2(\mathfrak{a}_t:t\in[2^n,2^{n+1})\cap\II)^2\Big)^{1/2}}_{\ell^p(\ZZ^\Gamma)}.
\end{equation}
\end{proposition}
The next numerical inequality is crucial in our investigations. 
\begin{remark}[Rademacher--Menshov inequality]\label{rem:RM} 
Since all of the seminorms \eqref{supnorm}--\eqref{varnorm} are dominated by the $2$-variation seminorm we deduce that the Rademacher--Menshov inequality \cite[Lemma 2.5, p. 534]{MSZ2} holds for $\calS_p$. Let $b$ and $s$ be fixed positive integers. Then for any complex-valued sequence $(a_j: b\le j\le 2^s)$ one has
\begin{equation}\label{eq:remark3}
    \calS_p(a_n:b\le n\le 2^s)\le\sqrt{2}\norm[\Big]{\sum_{i=1}^s\Big(\sum_{j}{|a_{u_{j+1}^i}-a_{u_j^i}|^2}\Big)^{1/2}}_{\ell^p(\ZZ^\Gamma)},
\end{equation}
where $[u_j^i,u_{j+1}^i)$ are dyadic intervals of the form $[j2^i,(j+1)2^{i})$ for some $0\le i\le s$, $0\le j\le 2^{s-i}-1$, contained in $[b, 2^s]$ (in particular, the number of intervals occurring in the inner sum is finite).
\end{remark}
\section{Proof of Theorem~\ref{Thm1:bootstrap}}
Assume that $p\in(1,\infty)$ and let $f\in \ell^p(\ZZ^\Gamma)$ be a compactly supported function. Let $\mathbb{U}:=\bigcup_{n\in\ZZ}2^n\NN$. Let us note that it is enough to establish the following inequality
\begin{equation}\label{eq:toprove}
    \calS_p(M_t f:t\in\mathbb{U})\lesssim_{\calS_p}\|f\|_{\ell^p(\ZZ^\Gamma)},
\end{equation}
where the implied constant may depend on the seminorm $\calS_p$ and $p\in(1,\infty)$ but is independent of $f$. Let us choose $p_0\in(1,2)$, close to $1$ such that $p\in(p_0,p_0')$. Note that $M_t f\equiv f$ for $t\in(0,1)$. By Proposition~\ref{prop:longandshort} we can split \eqref{eq:toprove} into long ``jumps'' and short variations 
\begin{equation}\label{eq:remark2}
     \calS_p(M_t f:t\in\mathbb{U})\lesssim\calS_p(M_{2^{n}}f:n\in\NN_0)+\norm[\Big]{\Big(\sum_{l=0}^\infty V^2\big(M_{n}f:n\in[2^{l},2^{l+1}]\cap\mathbb{U}\big)^2\Big)^{1/2}}_{\ell^p(\ZZ^\Gamma)}.
\end{equation}
We will estimate separately each part of the right hand side of \eqref{eq:remark2}.
\subsection{Estimates for the long jumps}\label{sec:longjumps}
The aim of this subsection is to give a proof of the estimate for the long jumps,
\begin{equation}\label{longjumps}
    \calS_p(M_{2^n} f: n\in\NN_0)\lesssim_{\calS_p}\|f\|_{\ell^p(\ZZ^\Gamma)},
\end{equation}
where the implicit constant may only depend on the seminorm $\calS_p$, but is independent of $f$. For this purpose we will exploit the following bootstrap argument. For $N\in\NN$ let us consider the following cut-off seminorms
\begin{equation}\label{poprawki:eq1}
    \calS_p(M_{2^n} f: n\in[0,N]\cap\NN_0).
\end{equation}
By $R_{p}(N)$ we denote the smallest constant $C>0$ for which the following estimate holds
\begin{equation*}
   \calS_p(M_{2^n} f: n\in[0,N]\cap\NN_0)\le C\|f\|_{\ell^p(\ZZ^\Gamma)},\quad f\in \ell^p(\ZZ^\Gamma).
\end{equation*}
By the estimate \eqref{square} we know that $R_{p}(N)\lesssim N<\infty$. However, we will show that there exists a constant $C_{p}>0$ such that $R_{p}(N)\le C_{p}$ for any $N\in\N$. If such a constant exists, then by taking limit as $N\to\infty$ and by using the monotone convergence theorem one easily obtains \eqref{longjumps}. Without loss of generality we can assume that $R_{p}(N)>1$ and $N\in\N$ is large.

We start with writing operator $M_{2^n}$ in terms of the discrete Fourier transform. One has
\begin{equation*}
    M_{2^n}f(x)=T_{\ZZ^\Gamma}[m_{2^n}]f(x),\quad x\in\ZZ^\Gamma,
\end{equation*}
where
\begin{equation*}
    m_N(\xi):=\frac{1}{|\Omega_{N}\cap\ZZ^k|}\sum_{y\in\Omega_{N}\cap\ZZ^k}{e((y)^\Gamma\cdot\xi)},\quad N\in\NN_0,\xi\in\TT^\Gamma.
\end{equation*}
In order to prove \eqref{longjumps} we require several appropriately chosen parameters. Let $\alpha>0$ be such that
\begin{equation}\label{alfa}
    \alpha>100\left(\frac{1}{p_0}-\frac{1}{2}\right)\left(\frac{1}{p_0}-\frac{1}{\min\{p,p'\}}\right)^{-1}.
\end{equation}
Fix $\chi\in(0,1/10)$ and let $u\in\NN$ be a large natural number which will be specified later. Now we are able to define the so-called Ionsescu--Wainger projections. For this purpose, we introduce a diagonal matrix $A$ of size $|\Gamma| \times |\Gamma|$  such that
$(A v)_\gamma := \abs{\gamma} v_\gamma$ for any $\gamma \in \Gamma$ and $v\in\RR^\Gamma$, and for any $t > 0$ we also define
corresponding dilations by setting $t^A x=\big(t^{|\gamma|}x_{\gamma}: \gamma\in \Gamma\big)$ for
every $x\in\RR^\Gamma$. Let $\eta\colon\RR^\Gamma\to[0,1]$ be a smooth function such that
\begin{equation}\label{bumpfun}
    \eta(x)=\begin{cases}
    1, & |x|\le1/(16|\Gamma|), \\
    0, & |x|\ge1/(8|\Gamma|).\end{cases}
\end{equation}
Let us set $\varrho:=(10u)^{-1}$ and recall the family of rational fractions $\Sigma_{\leq n^u}$ related to the parameter $\varrho$ described in Section~\ref{sec:IW}. For each $n\in\NN$ we define the following function
\begin{equation}\label{IWproj}
    \Xi_n(\xi):=\sum_{a/q\in\Sigma_{\leq n^u}}\eta^2\big(2^{n(A-\chi I)}(\xi-a/q)\big),
\end{equation}
where $I$ is the $|\Gamma|\times |\Gamma|$ identity matrix. By Theorem \ref{thm:IW-mult} we have that
\begin{equation}\label{IWprojin}
    \big\|T_{\ZZ^\Gamma}[\Xi_n]f)\big\|_{\ell^p(\ZZ^\Gamma)}\lesssim_{u,\,p}\log (n+1)\|f\|_{\ell^p(\ZZ^\Gamma)},
\end{equation}
since for large $n$ one has $\varepsilon_n=2^{-n(|\gamma|-\chi)}\le e^{-n^{1/5}}=e^{-(n^{u})^{2\varrho}}$. We use projections defined in \eqref{IWproj} to partition the multiplier $m_{2^n}$ and estimate \eqref{poprawki:eq1} by
\begin{align}
    \calS_p\big(T_{\ZZ^\Gamma}[m_{2^n}\Xi_n]f):n\in[0,N]\cap\NN_0\big)+\calS_p\big(T_{\ZZ^\Gamma}[(1-\Xi_n)m_{2^n}]f):n\in[0,N]\cap\NN_0\big).
\end{align}
Remark that the first and second term in the above inequality correspond to the major and minor arcs in the Hardy--Littlewood circle method.
\subsubsection{Estimates for the minor arcs}\label{Sec:longjumpsminor}
Now, our aim is to prove that
\begin{equation}\label{eq:minor-arcs}
   \calS_p\big(T_{\ZZ^\Gamma}[(1-\Xi_n)m_{2^n}]f):n\in[0,N]\cap\NN_0\big)\lesssim\|f\|_{\ell^p(\ZZ^\Gamma)}.
\end{equation}
We will obtain the above estimate by exploiting the standard approach--by using Weyl's inequality. Note that each seminorm \eqref{supnorm}--\eqref{varnorm} is bounded by the $2$-variation seminorm. Moreover, since the $r$-variation seminorms are non-increasing in $r$ we may estimate $V^2$ by the $V^1$ and hence 
\begin{align}\label{eq:bound minor-arcs}
    {\rm LHS}\eqref{eq:minor-arcs}\lesssim\sum_{n=0}^\infty\norm[\big]{T_{\ZZ^\Gamma}[(1-\Xi_n)m_{2^n}]f}_{\ell^p(\ZZ^\Gamma)}.
\end{align}
Therefore, it is enough to show
\begin{equation}\label{ex1}
    \norm[\big]{T_{\ZZ^\Gamma}[(1-\Xi_n)m_{2^n}]f}_{\ell^p(\ZZ^\Gamma)}\lesssim(n+1)^{-2}\|f\|_{\ell^p(\ZZ^\Gamma)}.
\end{equation}
In order to prove the estimate above we appeal to Weyl's inequality and proceed as in  \cite[Lemma 3.29, p. 34]{MSZ3}.
\subsubsection{Major arcs and scale distinction}
Now we return to the major arcs. Let $\tilde{\eta}(x)=\eta(x/2)$. We define new multipliers by setting
\begin{equation*}
    \Xi_n^s(\xi):=\sum_{a/q\in\Sigma_{s^u}}\eta^2\big(2^{n(A-\chi I)}(\xi-a/q)\big)\tilde{\eta}^2\big(2^{s(A-\chi I)}(\xi-a/q)\big),
\end{equation*}
where $\Sigma_{s^u}=\Sigma_{\leq(s+1)^u}\setminus\Sigma_{\leq s^u}$ for $s\in\NN$ and $\Sigma_{0^u}=\Sigma_{\leq 1}$. It is easy to see that one has $\Xi_n(\xi)=\sum_{s=0}^{n-1}\Xi_n^s(\xi)$.
Next, by \eqref{eq:split2} we obtain the following estimate
\begin{align*}
    \calS_p\big(T_{\ZZ^\Gamma}[m_{2^n}\Xi_n]f):n\in[0,N]\cap\NN_0\big)\lesssim \sum_{s=0}^N\calS_p\big(T_{\ZZ^\Gamma}[m_{2^n}\Xi_n^s]f):n\in[s,N]\cap\NN_0\big)+\norm[\big]{T_{\ZZ^\Gamma}[m_{2^s}\Xi_s^s]f}_{\ell^p(\ZZ^\Gamma)}.
\end{align*}
Now, for $s\in\NN_0$ we set $\kappa_s=20d\ceil{(s+1)^{1/10}}$ and by \eqref{eq:split1} we see that the expression under the sum is bounded by
\begin{align*}
   \calS_p\big(T_{\ZZ^\Gamma}[m_{2^n}\Xi_n^s]f):n\in[s,N]\cap\NN_0,n\leq 2^{\kappa_s+1}\big)&+\calS_p\big(T_{\ZZ^\Gamma}[m_{2^n}\Xi_n^s]f):n\in\NN_0,n>2^{\kappa_s}\big)\\
   &+\norm[\big]{T_{\ZZ^\Gamma}[m_{2^s}\Xi_s^s]f}_{\ell^p(\ZZ^\Gamma)}.
\end{align*}
The first term corresponds to small scales and the second one to large scales. For $p\in(1,\infty)$ we will show the following bounds:
\begin{align}
    \big\|T_{\ZZ^\Gamma}[m_{2^s}\Xi_s^s]f\big\|_{\ell^p(\ZZ^\Gamma)}&\lesssim(s+1)^{-3}\|f\|_{\ell^p(\ZZ^\Gamma)},\label{pinf1}\\
    \calS_p\big(T_{\ZZ^\Gamma}[m_{2^n}\Xi_n^s]f:n\in\NN_0,n>2^{\kappa_s}\big)&\lesssim(s+1)^{-3}\|f\|_{\ell^p(\ZZ^\Gamma)}.\label{pinf3}
\end{align}
Moreover, in the case of $p\in(1,2]$, we prove that the inequality
\begin{equation}
    \calS_p\big(T_{\ZZ^\Gamma}[m_{2^n}\Xi_n^s]f):n\in[s,N]\cap\NN_0,n\leq 2^{\kappa_s+1}\big)\lesssim R_{p}(N)^{\beta(p)}(s+1)^{-3}\|f\|_{\ell^p(\ZZ^\Gamma)},\label{pinf2}
\end{equation}
holds with some $\beta(p)\in[0,1)$. If we show the above inequalities, then for $p\in(1,2]$ the inequality \eqref{eq:minor-arcs} gives us
\begin{equation*}
   R_{p}(N)\lesssim 1+\sum_{s=0}^\infty{(s+1)^{-3}(R_{p}(N)^{\beta(p)}+1)}\lesssim R_{p}(N)^{\beta(p)},
\end{equation*}
since $R_{p}(N)\ge1$ and $\beta(p)\in[0,1)$. This gives $R_{p}(N)\lesssim_{p} 1$ and thus the proof is complete in the case of $p\in(1,2]$. When $p\in(2,\infty)$ a minor change is required for the estimate (\ref{pinf2}). Namely, in this case we show that
\begin{equation}
     \calS_p\big(T_{\ZZ^\Gamma}[m_{2^n}\Xi_n^s]f):n\in[s,N]\cap\NN,n\leq 2^{\kappa_s+1}\big)\lesssim R_{p'}(N)^{\beta'(p)}(s+1)^{-3}\|f\|_{\ell^p(\ZZ^\Gamma)},\label{pinf4}
\end{equation}
where $1/p+1/p'=1$ and $\beta'(p)\in(0,1)$. Since $p'\in(1,2)$, by the first part one has that $R_{p'}(N)\lesssim_{p'} 1$ and consequently
\begin{equation*}
    R_{p}(N)\lesssim 1+\sum_{s=0}^\infty{(s+1)^{-3}(R_{p'}(N)^{\beta'(p)}+1)}\lesssim_{p}1
\end{equation*}
which finishes the proof in the case of $p\in(2,\infty)$.
\subsubsection{Estimate for \eqref{pinf1}}\label{sec:phiandgauss}
We prove the estimate \eqref{pinf1} by approximating the multiplier $m_{2^s}\Xi_s^s$ by a suitably chosen integral. At first, note that by Theorem~\ref{thm:IW-mult}, for any $p\in(1,\infty)$, we obtain the estimate
\begin{equation}\label{inter1}
\norm[\big]{T_{\ZZ^\Gamma}[m_{2^s}\Xi_s^s]f}_{\ell^p(\ZZ^\Gamma)}\lesssim\log(s+1)\|f\|_{\ell^p(\ZZ^\Gamma)}.
\end{equation}
For further reference, for $N\in\NN_0$ we denote
\begin{align*}
    \Phi_N(\xi)&:=\frac{1}{|\Omega_{N}|}\int_{\Omega_{N}}{e(\xi\cdot(t)^\Gamma){\rm d}t}\quad{\rm and}\quad G(a/q):=\frac{1}{q^k}\sum_{r\in\N_q^k}{e((a/q)\cdot(r)^\Gamma)}.
\end{align*}
We have the following estimates for the function $\Phi_N$:
\begin{equation}\label{Phi1}
    |\Phi_N(\xi)|\lesssim |N^A\xi|_\infty^{-1/|\Gamma|}\quad\text{and}\quad|\Phi_N(\xi)-1|\lesssim |N^A\xi|_\infty,
\end{equation}
where the first estimate follows from the multidimensional van der Corput lemma \cite[Proposition 5, p. 342]{bigs} and the second one is a simple consequence of the mean value theorem. Note that by the multidimensional version of Weyl’s inequality \cite[Proposition 3]{SW2} one has
\begin{equation}\label{gaussum}
    |G(a,q)|\lesssim_k q^{-\delta}
\end{equation}
for some $\delta>0$.

By appealing to \cite[Proposition 4.18]{MSZ3} and proceeding as in \cite[Lemma 3.38, p. 36]{MSZ3} one can prove that
 \begin{equation}\label{reduction}
     (m_{2^s}\Xi_s^s)(\xi)=\sum_{a/q\in\Sigma_{s^u}}G(a/q)\Phi_{2^s}(\xi-a/q)\eta^2\big(2^{s(A-\chi I)}(\xi-a/q)\big)+\mathcal{O}(2^{-s/2}).
 \end{equation}
By Plancherel's theorem, the above estimate together with the estimate for the Gauss sum \eqref{gaussum} yield
\begin{equation}\label{inter2}
\begin{aligned}
    \norm[\big]{T_{\ZZ^\Gamma}[m_{2^s}\Xi_s^s]f}_{\ell^2(\ZZ^\Gamma)}\lesssim (s+1)^{-\alpha}\|f\|_{\ell^2(\ZZ^\Gamma)}
\end{aligned}
\end{equation}
provided that $u\in\NN$ satisfies $u>\alpha\delta^{-1}$. If we interpolate (\ref{inter1}) for $p=p_0$ with (\ref{inter2}) we get that (\ref{pinf1}) holds.
\subsubsection{Estimates for the large scales}
Now we focus on proving the estimate for the large scales,
\begin{equation*}
    \calS_p\big(T_{\ZZ^\Gamma}[m_{2^n}\Xi_n^s]f:n\in\NN_0,n>2^{\kappa_s}\big)\lesssim(s+1)^{-3}\|f\|_{\ell^p(\ZZ^\Gamma)}.
\end{equation*}
Again, by \cite[Proposition 4.18]{MSZ3} (compare with \eqref{reduction}) we have
\begin{align*}
    m_{2^n}\Xi_n^s(\xi)=h_s^n(\xi)+\mathcal{O}(2^{-n/2}),
\end{align*}
where
\begin{equation*}
    h_s^n(\xi):=\sum_{a/q\in\Sigma_{s^u}}{G(a,q)\Phi_{2^n}(\xi-a/q)\eta^2(2^{n(A-\chi I)}(\xi-a/q))\Tilde{\eta}^2(2^{s(A-\chi I)}(\xi-a/q))}
\end{equation*}
and we see that is suffices to prove that
\begin{equation}\label{msw9}
   \calS_p\big(T_{\ZZ^\Gamma}[h_n^s]f):n\in\NN_0,n>2^{\kappa_s}\big)\lesssim(s+1)^{-3}\|f\|_{\ell^p(\ZZ^\Gamma)}.
\end{equation}
Proceeding as in \cite[Section 3.6, pp. 40--41]{MSZ3} and invoking transference principle \cite[Corollary 2.1, p. 196]{MSW} we conclude that the estimate \eqref{msw9} follows by continuous seminorm inequalities. In the case of the jump seminorm we use \cite[Theorem 1.3]{MSZ1}.
\subsubsection{Small scales and the discrete Littlewood--Paley theory: estimates for \eqref{pinf2} and \eqref{pinf4}}\label{sec:SsDPL}
We begin with writing inequalities \eqref{pinf2} and \eqref{pinf4} in a more convenient form, namely
\begin{equation}\label{eq:sq1}
    \calS_p\big(T_{\ZZ^\Gamma}[m_{2^n}\Xi_n^s]f):n\in[s,N]\cap\NN_0,n\leq 2^{\kappa_s+1}\big)\lesssim B_p(N)(s+1)^{-3}\|f\|_{\ell^p(\ZZ^\Gamma)},
\end{equation}
where for $p\in(1,2]$ the constant $B_p(N)=R_p(N)^{\beta(p)}$ and for $p\in(2,\infty)$ we have $B_p(N)=R_{p'}(N)^{\beta'(p)}$ with $\beta(p),\beta'(p)\in[0,1)$. By Remark~\ref{rem:RM} we can apply the Rademacher--Menshov inequality \eqref{eq:remark3} and estimate
\begin{equation}\label{eq:sq2}
   {\rm LHS}\eqref{eq:sq1}\lesssim\sum_{i=0}^{\kappa_s+1}\norm[\Big]{\Big(\sum_{j=0}^{2^{\kappa_s+1-i}-1}\big|\sum_{n\in I_j^i}T_{\ZZ^\Gamma}\big[(m_{2^{n+1}}\Xi_{n+1}^s-m_{2^n}\Xi_n^s)\big]f\big|^2\Big)^{1/2}}_{\ell^p(\ZZ^\Gamma)},
\end{equation}
where $I_j^i:=[j2^i,(j+1)2^i)\cap [s,\min\{N,2^{\kappa_s+1}\})\cap\NN_0$ since the inner sum telescopes. Now, by the triangle inequality one has
\begin{align}
{\rm RHS}\eqref{eq:sq2}&\le\sum_{i=0}^{\kappa_s+1}\norm[\Big]{\Big(\sum_{j}\big|\sum_{n\in I_j^i}T_{\ZZ^\Gamma}\big[(m_{2^{n+1}}-m_{2^n})\Xi_n^s\big]f\big|^2\Big)^{1/2}}_{\ell^p(\ZZ^\Gamma)}\label{kin1}\\
    &\,\,+\sum_{i=0}^{\kappa_s+1}\norm[\Big]{\Big(\sum_{j}\big|\sum_{n\in I_j^i}T_{\ZZ^\Gamma}\big[m_{2^{n+1}}(\Xi_{n+1}^s-\Xi_n^s)\big]f\big|^2\Big)^{1/2}}_{\ell^p(\ZZ^\Gamma)}\label{kin2}.
\end{align}
Here and later on we will omit the limits of summation in $j$ for the sake of clarity. Now we invoke Khintchine’s inequality to \eqref{kin1} and \eqref{kin2} and as a consequence we see that the estimate \eqref{eq:sq1} will follow if we show that inequalities
\begin{align}
    \norm[\Big]{\sum_{j}\sum_{n\in I_j^i}T_{\ZZ^\Gamma}\big[\varepsilon_j(m_{2^{n+1}}-m_{2^n})\Xi_n^s\big]f}_{\ell^p(\ZZ^\Gamma)}&\lesssim B_p(N)(s+1)^{-5}\|f\|_{\ell^p(\ZZ^\Gamma)},\label{kin5}\\
    \norm[\Big]{\sum_{j}\sum_{n\in I_j^i}T_{\ZZ^\Gamma}\big[\varepsilon_j m_{2^{n+1}}(\Xi_{n+1}^s-\Xi_n^s)]f}_{\ell^p(\ZZ^\Gamma)}&\lesssim(s+1)^{-5}\|f\|_{\ell^p(\ZZ^\Gamma)},\label{kin6}
\end{align}
hold for every $i\le\kappa_s+1$ and any sequence $(\varepsilon_j\colon j\le 2^{\kappa_s+1-i}-1)\subseteq\{-1,1\}$. Finally, for estimates \eqref{kin5} and \eqref{kin6} it is enough to show that for any interval $I\subseteq[s,\min\{N,2^{\kappa_s+1}\})\cap\NN$ and for any sequence $(\varepsilon_n\colon n\in I)\subseteq\{-1,1\}$ one has
\begin{align}
    \norm[\Big]{\sum_{n\in I}T_{\ZZ^\Gamma}\big[\varepsilon_n(m_{2^{n+1}}-m_{2^n})\Xi_n^s]f}_{\ell^p(\ZZ^\Gamma)}&\lesssim B_p(N)(s+1)^{-5}\|f\|_{\ell^p(\ZZ^\Gamma)},\label{kin7}\\
    \norm[\Big]{\sum_{n\in I}T_{\ZZ^\Gamma}\big[\varepsilon_n m_{2^{n+1}}(\Xi_{n+1}^s-\Xi_n^s)\big]f}_{\ell^p(\ZZ^\Gamma)}&\lesssim(s+1)^{-5}\|f\|_{\ell^p(\ZZ^\Gamma)}.\label{kin8}
\end{align}
To prove \eqref{kin8}, by the triangle inequality, it is enough to establish
\begin{equation}\label{i3}
    \big\|T_{\ZZ^\Gamma}[m_{2^{n+1}}(\Xi_{n+1}^s-\Xi_n^s)]f\big\|_{\ell^p(\ZZ^\Gamma)}\lesssim (s+1)^{-5}(n+1)^{-3}\|f\|_{\ell^p(\ZZ^\Gamma)}.
\end{equation}
For any $p\in(1,\infty)$ by Theorem~\ref{thm:IW-mult} we have
\begin{equation}\label{i1}
    \big\|T_{\ZZ^\Gamma}[m_{2^{n+1}}(\Xi_{n+1}^s-\Xi_n^s)]f\big\|_{\ell^p(\ZZ^\Gamma)}\lesssim\log(s+1)\|f\|_{\ell^p(\ZZ^\Gamma)}.
\end{equation}
In the case of $p=2$ we use \cite[Proposition 4.18]{MSZ3} (compare with \eqref{reduction}, the only difference is the error term which is a consequence of the inequality $2^{-n/2}\lesssim2^{-(n+s)/4}$ since $n\ge s$) to get
\begin{equation}\label{eq:approx1}
    m_{2^{n+1}}(\xi)=G(a/q)\Phi_{2^{n+1}}(\xi-a/q)+\mathcal{O}(2^{-(n+s)/4}),
\end{equation}
where $a/q$ is the rational approximation of $\xi$ such that for every $\gamma\in\Gamma$ holds $|\xi_\gamma-a_\gamma/q|\lesssim 2^{-n(|\gamma|-\chi)}$. Next, we note that the function $\Xi_{n+1}^s-\Xi_n^s$ is nonzero only for $\xi$ such that $|2^{(n+1)(A-\chi I)}(\xi-a/q)|\gtrsim1$ and $|2^{n(A-\chi I)}(\xi-a/q)|\lesssim1$, where $a/q\in\Sigma_{s^u}$. Hence, by Plancherel's theorem the estimate \eqref{gaussum} and by the van der Corput estimate in \eqref{Phi1} one gets
\begin{equation*}
    \big\|T_{\ZZ^\Gamma}[m_{2^{n+1}}(\Xi_{n+1}^s-\Xi_n^s)]f\big\|_{\ell^2(\ZZ^\Gamma)}\lesssim(s+1)^{-\alpha}2^{-n\chi/|\Gamma|}\|f\|_{\ell^2(\ZZ^\Gamma)}.
\end{equation*}
Interpolating the above inequality with \eqref{i1} for $p=p_0$ yields \eqref{i3}.\\
\indent Now we focus our attention on the proof of the estimate \eqref{kin7}. For this purpose we introduce new multipliers of the form
\begin{equation*}
    \Xi_n^{s,j}(\xi):=\sum_{a/q\in\Sigma_{s^u}}\eta^2\big(2^{nA+jI}(\xi-a/q)\big)\tilde{\eta}^2\big(2^{s(A-\chi I)}(\xi-a/q)\big),\quad j\in\ZZ.
\end{equation*}
We have the following decomposition
\begin{equation*}
    \Xi_n^s(\xi)=\sum_{-\floor{\chi n}\le j<n}\big(\Xi_n^{s,j}(\xi)-\Xi_n^{s,j+1}(\xi)\big)+\big(\Xi_n^{s,-\chi n}(\xi)-\Xi_n^{s,-\floor{\chi n}}(\xi)\big)+\Xi_n^{s,n}(\xi),
\end{equation*}
since the sum above telescopes. By using the new multipliers one may write
\begin{align*}
    {\rm LHS}\eqref{kin7}&\le\big\|\sum_{n\in I}\sum_{-\floor{\chi n}\le j<n}T_{\ZZ^\Gamma}\big[\varepsilon_n(m_{2^{n+1}}-m_{2^n})(\Xi_n^{s,j}-\Xi_n^{s,j+1})\big]f\big\|_{\ell^p(\ZZ^\Gamma)}\\
    &\,\,+\big\|\sum_{n\in I}T_{\ZZ^\Gamma}\big[\varepsilon_n(m_{2^{n+1}}-m_{2^n})\big((\Xi_n^{s,-\chi n}-\Xi_n^{s,-\floor{\chi n}})+\Xi_n^{s,n}\big)\big]f\big\|_{\ell^p(\ZZ^\Gamma)}.
\end{align*}
Consequently, to obtain \eqref{kin7} it is enough to show two inequalities:
\begin{equation}\label{eq:sq3}
    \big\|\sum_{n\in I}\sum_{-\floor{\chi n}\le j<n}T_{\ZZ^\Gamma}\big[\varepsilon_n(m_{2^{n+1}}-m_{2^n})(\Xi_n^{s,j}-\Xi_n^{s,j+1})\big]f\big\|_{\ell^p(\ZZ^\Gamma)}\lesssim B_p(N)(s+1)^{-5}\|f\|_{\ell^p(\ZZ^\Gamma)}
\end{equation}
and
\begin{equation}\label{eq:sq4}
    \big\|\sum_{n\in I}T_{\ZZ^\Gamma}\big[\varepsilon_n(m_{2^{n+1}}-m_{2^n})\big((\Xi_n^{s,-\chi n}-\Xi_n^{s,-\floor{\chi n}})+\Xi_n^{s,n}\big)\big]f\big\|_{\ell^p(\ZZ^\Gamma)}\lesssim(s+1)^{-5}\|f\|_{\ell^p(\ZZ^\Gamma)}.
\end{equation}
In order to show \eqref{eq:sq4} we apply the triangle inequality and we see it will follow from
\begin{equation*}
     \big\|T_{\ZZ^\Gamma}\big[(m_{2^{n+1}}-m_{2^n})\big((\Xi_n^{s,-\chi n}-\Xi_n^{s,-\floor{\chi n}})+\Xi_n^{s,n}\big)\big]f\big\|_{\ell^p(\ZZ^\Gamma)}\lesssim(s+1)^{-5}(n+1)^{-3}\|f\|_{\ell^p(\ZZ^\Gamma)}.
\end{equation*}
By using Theorem~\ref{thm:IW-mult}, estimates \eqref{Phi1} and \eqref{gaussum} one can prove the above inequality in the same way as \eqref{i3} -- we omit the details.\\
\indent Now we may return to \eqref{eq:sq3}. Let $I_j:=\big\{n\in\NN_0\colon n\in I,n\ge\max\{-j/\chi,\,j-1\}\big\}$. By changing the order of summation, we see that to prove \eqref{eq:sq3} it is enough to show that
\begin{equation}\label{ex6.1}
    \Big\|\sum_{n\in I_j}
    T_{\ZZ^\Gamma}\big[\varepsilon_n(m_{2^{n+1}}-m_{2^n})(\Xi_n^{s,j}-\Xi_n^{s,j+1})\big]f\Big\|_{\ell^p(\ZZ^\Gamma)}\lesssim (s+1)^{-5}B_p(N)2^{-|j|\beta}\|f\|_{\ell^p(\ZZ^\Gamma)}
\end{equation}
holds for some $\beta=\beta_p>0$. Remark that one has
\begin{equation*}
    \left(\Xi_n^{s,j}-\Xi_n^{s,j+1}\right)(\xi)=\Delta_{n,\,s}^{j,\,1}(\xi)\Delta_{n,\,s}^{j,\,2}(\xi),
\end{equation*}
where
\begin{align*}
    \Delta_{n,\,s}^{j,\,1}(\xi)&:=\sum_{a/q\in\Sigma_{s^u}}\Big[\eta\left(2^{nA+(j-1)I}(\xi-a/q)\right)-\eta\left(2^{nA+(j+2)I}(\xi-a/q)\right)\Big]\Tilde{\eta}(2^{s(A-\chi I)}(\xi-a/q)),\\
    \Delta_{n,\,s}^{j,\,2}(\xi)&:=\sum_{a/q\in\Sigma_{s^u}}\Big[\eta^2\big(2^{nA+jI}(\xi-a/q)\big)-\eta^2\big(2^{nA+(j+1)I}(\xi-a/q)\big)\Big]\Tilde{\eta}(2^{s(A-\chi I)}(\xi-a/q)).
\end{align*}
Now we will derive from the discrete Littlewood--Paley theory which originates in \cite[Theorem 3.3]{M10}. Let $j,n\in\ZZ$ and let $\Phi_{j,\,n}(\xi)=\Phi(2^{nA+jI}\xi)$, where $\Phi$ is a Schwartz function such that $\Phi(0)=0$. Observe that one has
\begin{equation*}
    |\Phi_{n,j}(\xi)|\lesssim\min\{|2^{nA+jI}\xi|,|2^{nA+jI}\xi|^{-1}\}.
\end{equation*}
Moreover, for any $p\in(1,\infty)$ there is a constant $C_p>0$ such that
\begin{equation*}
    \norm[\big]{\sup_{n\in\ZZ}|T_{\RR^\Gamma}[|\Phi_{n,j}|]f|}_{L^p(\RR^\Gamma)}\le C_p\|f\|_{L^p(\RR^\Gamma)}.
\end{equation*}
Hence, by \cite[Theorem B]{DR86} for any $-\infty\le M_1\le M_2\le\infty$ we have
\begin{equation*}
    \norm[\big]{\big(\sum_{M_1\le n\le M_2}|T_{\RR^\Gamma}[\Phi_{j,n}]f|^2\big)^{1/2}}_{L^p(\RR^\Gamma)}\lesssim_p\|f\|_{L^p(\RR^\Gamma)},
\end{equation*}
where the implied constant is independent of $j, M_1$ and $M_2$. Therefore, by Theorem~\ref{thm:IW-mult} the multiplier
\begin{equation}\label{eq:singulardiscrete}
    \Omega_s^{j,\,n}(\xi):=\sum_{a/q\in\Sigma_{s^u}}{\Phi_{j,\,n}(\xi-a/q)\Tilde{\eta}(2^{s(A-\chi I)}(\xi-a/q))}
\end{equation}
satisfies
\begin{equation}\label{eq:DPL1}
    \Big\|\big(\sum_{M_1\le n\le M_2}|T_{\ZZ^\Gamma}[\Omega_s^{j,\,n}]f|^2\big)^{1/2}\Big\|_{\ell^p(\ZZ^\Gamma)}\lesssim\log(s+1)\|f\|_{\ell^p(\ZZ^\Gamma)}.
\end{equation}
Moreover, if $\Phi$ is a real valued function, then the dual version of the inequality \eqref{eq:DPL1} also holds, namely
\begin{equation}\label{eq:DPL2}
    \big\|\sum_{M_1\le n\le M_2}|T_{\ZZ^\Gamma}[\Omega_s^{j,\,n}]f_n)\big\|_{\ell^p(\ZZ^\Gamma)}\lesssim\log (s+1)\Big\|\big(\sum_{M_1\le n\le M_2}{|f_n|^2}\big)^{1/2}\Big\|_{\ell^p(\ZZ^\Gamma)},
\end{equation}
where $(f_n: M_1\le n\le M_2)$ is a sequence of functions such that
\begin{equation*}
    \norm[\big]{\big(\sum_{M_1\le n\le M_2}|f_n|^2\big)^{1/2}}_{\ell^p(\ZZ^\Gamma)}<\infty.
\end{equation*}
It is easy to see that multipliers $\Delta_{n,\,s}^{j,\,1}$ and $\Delta_{n,\,s}^{j,\,2}$ can be written as \eqref{eq:singulardiscrete}. Hence, by applying the inequality \eqref{eq:DPL2} to the multiplier $\Delta_{n,\,s}^{j,\,1}$ we get 
\begin{align*}
   {\rm LHS}\eqref{ex6.1}\lesssim\log(s+1)\Big\|\big(\sum_{n\in I_j}|T_{\ZZ^\Gamma}[(m_{2^{n+1}}-m_{2^n})\Delta_{n,\,s}^{j,\,2}]f|^2\big)^{1/2}\Big\|_{\ell^p(\ZZ^\Gamma)}.
\end{align*}
Consequently, the estimate \eqref{ex6.1} will follow if we prove that
\begin{equation}\label{ex6}
    \Big\|\big(\sum_{n\in I_j}|T_{\ZZ^\Gamma}[(m_{2^{n+1}}-m_{2^n})\Delta_{n,\,s}^{j,\,2}]f|^2\big)^{1/2}\Big\|_{\ell^p(\ZZ^\Gamma)}\lesssim(s+1)^{-10}B_p(N)2^{-|j|\beta}\|f\|_{\ell^p(\ZZ^\Gamma)},
\end{equation}
for any $p\in(1,\infty)$.
\subsubsection{Bootstrap estimates for the square function in \eqref{ex6}}\label{sec:bootstrapsquare}
We start with proving some estimates in the case of $p=2$. For simplicity, we denote $$\psi_{n,\,j}^{a/q,\,s}(\xi):=\big[\eta^2\big(2^{nA+jI}(\xi-a/q)\big)-\eta^2\big(2^{nA+(j+1)I}(\xi-a/q)\big)\big]\Tilde{\eta}\big(2^{s(A-\chi I)}(\xi-a/q)\big).$$ Observe that $\psi_{n,\,j}^{a/q,\,s}$ is nonzero only if $2^{-(j+2)}\lesssim|2^{nA}(\xi-a/q)|_\infty\lesssim2^{-j}$. Hence $|\xi_\gamma-a_\gamma/q|\le2^{-n(|\gamma|-\chi)}$ for $\gamma\in\Gamma$ since $n\ge-j/\chi$. By \cite[Proposition 4.18]{MSZ3} we have
\begin{equation}\label{eq1:bsql2}
    m_{2^n}(\xi)=G(a,q)\Phi_{2^n}(\xi-a/q)+\mathcal{O}(2^{-n/2}),
\end{equation}
where $a/q$ satisfy $|\xi_\gamma-a_\gamma/q|\lesssim 2^{-n(|\gamma|-\chi)}$ for every $\gamma\in\Gamma$. To simplify notation we denote $w_n(\xi):=\min\{|2^{nA}\xi|_\infty,|2^{nA}\xi|_\infty^{-1/|\Gamma|}\}$. 
By using estimates \eqref{eq1:bsql2}, \eqref{Phi1} and \eqref{gaussum} one may conclude that
\begin{align*}
    |(m_{2^{n+1}}-m_{2^n})(\xi)|&\lesssim (s+1)^{-u\delta}w_n(\xi-a/q)+(s+1)^{-u\delta}\mathcal{O}(2^{-n/4}),\\
     |(m_{2^{n+1}}-m_{2^n})(\xi)|\psi_{n,\,j}^{a/q,\,s}(\xi)&\lesssim (s+1)^{-u\delta}2^{-|j|/d}+(s+1)^{-u\delta}\mathcal{O}(2^{-n/4}).
\end{align*}
Hence, by using the above estimates we obtain
\begin{align*}
    &\sum_{n\in I_j}\sum_{a/q\in\Sigma_{s^u}}|(m_{2^{n+1}}-m_{2^n})(\xi)\psi_{n,\,j}^{a/q,\,s}(\xi)|^2\lesssim(s+1)^{-2u\delta}2^{-|j|\beta},
\end{align*}
 for some $\beta>0$. Hence, by Plancherel's theorem 
\begin{equation}\label{m1}
\begin{aligned}
    \Big\|\big(\sum_{n\in I_j}|T_{\ZZ^\Gamma}[(m_{2^{n+1}}-m_{2^n})\Delta_{n,\,s}^{j,\,2}]f|^2\big)^{1/2}\Big\|_{\ell^2(\ZZ^\Gamma)}&\lesssim(s+1)^{-\delta u}2^{-|j|\beta/2}\|f\|_{\ell^2(\ZZ^\Gamma)}.
\end{aligned}
\end{equation}
Let us note that the above estimate together with the estimate for the large scales \eqref{pinf3} proves that $R_2(N)<\infty$ for any $N\in\NN$ so we have proven estimate \eqref{longjumps} in the case of $p=2$. In order to handle other values of $p$ we will relate to the bootstrap argument from \cite{MSZ2}. This argument originates in \cite{NSW} and \cite{DR86}.
\begin{lemma}{\rm \cite[Lemma 2.8]{MSZ2}}\label{lemma28}
Suppose that $(X,\mathcal{B},m)$ is a $\sigma$-finite measure space and $(B_k)_{k\in\mathbb{J}}$ is a sequence of linear operators on $L^1(X)+L^\infty(X)$ indexed by a countable set $\mathbb{J}$. The corresponding maximal operator is defined by
\begin{equation*}
    B_{\ast,\,\mathbb{J}}f\colon=\sup_{k\in\mathbb{J}}\sup_{|g|\le|f|}|B_kg|,
\end{equation*}
where the supremum is taken in the lattice sense. Let $q_0,q_1\in[1,\infty]$ and $0\le\vartheta\le1$ with $\frac12=\frac{1-\vartheta}{q_0}$ and $q_0\le q_1$. Let $q_\vartheta\in[q_0,q_1]$ be given by $\frac{1}{q_\vartheta}=\frac{1-\vartheta}{q_0}+\frac{\vartheta}{q_1}=\frac12+\frac{1-q_0/2}{q_1}$. Then
\begin{equation*}
    \Big\| \big(\sum_{k\in\mathbb{J}}|B_kg_k|^2 \big)^{1/2}\Big\|_{L^{q_\vartheta}(X)}\le(\sup_{k\in\mathbb{J}}\|B_k\|_{L^{q_0}\to L^{q_0}})^{1-\vartheta}\|B_{\ast,\,\mathbb{J}}\|^\vartheta_{L^{q_1}\to L^{q_1}}\Big\| \big(\sum_{k\in\mathbb{J}}|g_k|^2 \big)^{1/2}\Big\|_{L^{q_\vartheta}(X)}.
\end{equation*}
\end{lemma}
\noindent Here we have to make some distinction between the considered seminorms \eqref{supnorm}--\eqref{varnorm}.\\
\indent \emph{Case of the supremum, oscillation and $r$-variational seminorm.} At first let $p\in(1,2]$. We will apply Lemma~\ref{lemma28} with the set $\mathbb{J}=I_j\subseteq[0,N)$, parameters $q_0=1$, $q_1=p$, $\vartheta=1/2$, operators $B_n=M_{2^{n+1}}-M_{2^{n}}$ and functions $g_n=T_{\ZZ^\Gamma}[\Delta_{n,s}^{j,2}]f$. Now, since the norm of the operator $M_{2^n}$ is uniformly bounded we see that  for every $q\in(1,\infty)$ one has
\begin{align*}
    \sup_{n< N}\|B_n\|_{\ell^q\to \ell^q}&\lesssim 1.
\end{align*}
When we consider the supremum seminorm,
\begin{equation*}
    \calS_p(M_{2^n} f: n\in[0,N]\cap\NN_0)=\norm[\big]{\sup_{0\le n\le N}|(M_{2^{n}}-M_1)f|}_{\ell^p(\ZZ^\Gamma)},
\end{equation*}
then it is easy to check that for every $q\in(1,\infty)$
\begin{align*}
    \|B_{\ast,n< N}\|_{\ell^q\to \ell^q}\lesssim R_{q}(N)
\end{align*}
since the norm of the operator $M_t$ is uniformly bounded. Similarly, in the case of the oscillation seminorm,
\begin{equation*}
    \calS_p(M_{2^n} f: n\in[0,N]\cap\NN_0)=\sup_{K\in\NN}\sup_{I\in\mathfrak{S}_K([0,N]\cap\NN_0)}\norm[\big]{O_{I,K}^2(M_{2^n}f:n\in[0,N]\cap\NN_0)}_{\ell^p(\ZZ^\Gamma)},
\end{equation*}
by the inequality \eqref{domsup2} we have that the estimate
\begin{align*}
    \|B_{\ast,n< N}\|_{\ell^q\to \ell^q}\lesssim R_{q}(N)
\end{align*}
 yields for any $q\in(1,\infty)$. Finally, in the case of $r$-variational seminorm,
\begin{equation*}
     \calS_p(M_{2^n} f: n\in[0,N]\cap\NN_0)=\norm[\big]{V^r(M_{2^n}f:n\in[0,N]\cap\NN_0)}_{\ell^p(\ZZ^\Gamma)},
\end{equation*}
by inequality \eqref{domsup1} we see that for any $r\in(2,\infty)$ the estimate 
\begin{align*}
    \|B_{\ast,n< N}\|_{\ell^q\to \ell^q}\lesssim R_{q}(N)
\end{align*}
also holds for $q\in(1,\infty)$. Hence, in the case of the supremum, oscillation and $r$-variational seminorm by Lemma~\ref{lemma28} we may write
\begin{equation}\label{m}
\begin{aligned}
   \Big\|\big(\sum_{n\in\mathbb{J}}| B_n T_{\ZZ^\Gamma}[\Delta_{n,\,s}^{j,\,2}]f|^2\big)^{1/2}\Big\|_{\ell^{q_{1/2}}(\ZZ^\Gamma)}&\lesssim R_{p}(N)^{1/2}\Big\|\big(\sum_{n\in\mathbb{J}}|T_{\ZZ^\Gamma}[\Delta_{n,\,s}^{j,\,2}]f|^2\big)^{1/2}\Big\|_{\ell^{q_{1/2}}(\ZZ^\Gamma)}\\
   &\lesssim R_{p}(N)^{1/2}\log(s+1)\|f\|_{\ell^{q_{1/2}}(\ZZ^\Gamma)},
\end{aligned}
\end{equation}
where in the last inequality we have used \eqref{eq:DPL1}. Since $q_{1/2}<p\le2$, there exists $t\in[0,1)$ such that $\frac{1}{p}=\frac{t}{q_{1/2}}+\frac{1-t}{2}$. If we use definition of $q_{1/2}$ from Lemma \ref{lemma28} we see that $t=2-p$. Hence, by interpolating \eqref{m1} with \eqref{m} one has
\begin{align*}
    \Big\|\big(\sum_{n\in\JJ}|B_n T_{\ZZ^\Gamma}[\Delta_{n,\,s}^{j,\,2}]f|^2\big)^{1/2}\Big\|_{\ell^p(\ZZ^\Gamma)}\lesssim(s+1)^{-u\delta(1-t)}2^{-|j|\frac{\beta}{2} (1-t)}R_{p}(N)^{(2-p)/2}(s+1)\|f\|_{\ell^p(\ZZ^\Gamma)}.
\end{align*}
Since $u\in\NN$ can be large, we get that \eqref{ex6} is satisfied with $B_p(N)=R_p(N)^{(2-p)/2}$. Hence, we see that for $p\in(1,2]$ the inequality \eqref{pinf2} holds with $\beta(p)=\frac{2-p}{2}\in[0,1)$.

Now let us assume that $p\in(2,\infty)$. Then one has $p'\in(1,2)$ and therefore by applying Lemma \ref{lemma28} with $q_0=1$, $q_1=p'$ and  $\vartheta=1/2$, and by inequality \eqref{eq:DPL1} we obtain
\begin{equation}\label{ma}
\begin{aligned}
   \Big\|\big(\sum_{n\in\mathbb{J}}|B_nT_{\ZZ^\Gamma}[\Delta_{n,\,s}^{j,\,2}]f|^2\big)^{1/2}\Big\|_{\ell^{q_{1/2}}(\ZZ^\Gamma)}\lesssim R_{p'}(N)^{1/2}\log(s+1)\|f\|_{\ell^{q_{1/2}}(\ZZ^\Gamma)}
\end{aligned}
\end{equation}
where $q_{1/2}=2p/(2p-1)$. Now, since $B_n=M_{2^{n+1}}-M_{2^n}$ is a convolution operator we see that by duality the inequality \eqref{ma} holds for $q_{1/2}'=2p$. Since $2<p<q_{1/2}'$ there exists $\tau\in[0,1)$ such that $\frac{1}{p}=\frac{\tau}{q_{1/2}'}+\frac{1-\tau}{2}$ and 
\begin{equation*}
    \tau=\frac{2-p}{1-p}.
\end{equation*}
Hence, by interpolating \eqref{m1} with \eqref{ma} for $q_{1/2}'$ we may write
\begin{align*}
    \Big\|\big(\sum_{n\in\JJ}|B_nT_{\ZZ^\Gamma}[\Delta_{n,\,s}^{j,\,2}]f|^2\big)^{1/2}\Big\|_{\ell^p(\ZZ^\Gamma)}\lesssim(s+1)^{-u\delta(1-\tau)}2^{-|j|\frac{\beta}{2}(1-\tau)}R_{p'}(N)^{\frac{2-p}{2(1-p)}}(s+1)\|f\|_{\ell^p(\ZZ^\Gamma)}.
\end{align*}
Since $u\in\NN$ can be large, we get that \eqref{ex6} is satisfied with $B_p(N)=R_{p'}(N)^{\frac{2-p}{2(1-p)}}$. Hence, we see that for $p\in(2,\infty)$ the inequality \eqref{pinf4} holds with $\beta'(p)=\frac{2-p}{2(1-p)}\in[0,1)$.\\

\indent \emph{Case of the jump quasi-seminorm} In the context of the jump quasi-seminorm we need to proceed in a slightly different way since in this case we do not have a pointwise estimate of the form \eqref{domsup1} or even an $\ell^p$-estimate of the form \eqref{domsup2}. Fortunately, for $r>2$ one has ``weak $\ell^p$''-estimate \eqref{domweak} for the $r$-variation which we will use at this moment. As mentioned before we have already proved that $R_2(N)\lesssim1$. Hence, we may assume that $p\in(1,2)$. Let us consider $\lambda\in(0,1)$ such that
\begin{equation}\label{eq:lambdacon}
    \lambda>\max\big\{0,\frac{4-3p}{(p-2)^2}\big\}.
\end{equation}
We are going to apply Lemma~\ref{lemma28} with parameters $q_0=1$, $q_1=\lambda p+(1-\lambda)2$, $\vartheta=1/2$, operators $B_n=M_{2^{n+1}}-M_{2^n}$ and functions $g_n=T_{\ZZ^\Gamma}[\Delta_{n,s}^{j,2}]f$. If $\lambda$ satisfy condition \eqref{eq:lambdacon}, then one has $q_{1/2}<p<q_1<2$. Furthermore, for any $q\in(1,\infty)$
\begin{align*}
    \sup_{n<N}\|B_n\|_{\ell^q\to \ell^q}&\lesssim 1\quad\text{and}\quad\|B_{\ast,n< N}\|_{\ell^q\to \ell^q}\lesssim \|V^3(M_{2^n} : n\in[0,N]\cap\NN_0)\|_{\ell^q\to\ell^q},
\end{align*}
where the last inequality follows by \eqref{domsup1}. By using the inequality \eqref{domweak} we get weak type estimates
\begin{align*}
    \norm[\big]{V^3(M_{2^n}f: n\in[0,N]\cap\NN_0)}_{\ell^{2,\,\infty}(\ZZ^\Gamma)}&\lesssim\|f\|_{\ell^2(\ZZ^\Gamma)},\\ \norm[\big]{V^3(M_{2^n} f : n\in[0,N]\cap\NN_0)}_{\ell^{p,\,\infty}(\ZZ^\Gamma)}&\lesssim R_p(N)\|f\|_{\ell^p(\ZZ^\Gamma)}.
\end{align*}
Since $p<q_1<2$, one may use Marcinkiewicz's interpolation theorem to get
\begin{equation*}
    \|B_{\ast,n< N}\|_{\ell^{q_1}\to \ell^{q_1}}\lesssim_p R_p(N)^{\frac{p(2-q_1)}{q_1(2-p)}}.
\end{equation*}
Therefore, by Lemma~\ref{lemma28} and inequality \eqref{eq:DPL1} one has
\begin{equation}\label{j1}
\begin{aligned}
   \Big\|\big(\sum_{n\in\JJ}|B_n T_{\ZZ^\Gamma}[\Delta_{n,\,s}^{j,\,2}]f|^2\big)^{1/2}\Big\|_{\ell^{q_{1/2}}(\ZZ^\Gamma)}\lesssim R_p(N)^{\frac{p(2-q_1)}{2q_1(2-p)}}\log(s+1)\|f\|_{\ell^{q_{1/2}}(\ZZ^\Gamma)},
\end{aligned}
\end{equation}
Since $q_{1/2}<p<2$, there exists $t\in(0,1)$ such that $\frac{1}{p}=\frac{t}{q_{1/2}}+\frac{1-t}{2}$. Hence, by the definition of $q_{1/2}$ one has
\begin{equation*}
    t\frac{p(2-q_1)}{2q_1(2-p)}=\frac{2-q_1}{2}.
\end{equation*}
Interpolating (\ref{m1}) with \eqref{j1} leads to
\begin{align*}
    \Big\|\big(\sum_{n\in\JJ}|B_n T_{\ZZ^\Gamma}[\Delta_{n,\,s}^{j,\,2}]f|^2\big)^{1/2}\Big\|_{\ell^p(\ZZ^\Gamma)}\lesssim(s+1)^{-u\delta(1-t)}2^{-|j|\beta(1-t)/2}R_{p}(N)^{\frac{2-q_1}{2}}\log(s+1)\|f\|_{\ell^p(\ZZ^\Gamma)}.
\end{align*}
Since $u\in\NN$ can be large, we get that \eqref{ex6} in the case of the jump quasi-seminorm is satisfied with $B_p(N)=R_p(N)^{(2-q_1)/2}$. Hence, we see that for $p\in(1,2]$ the inequality \eqref{pinf2} holds with $\beta(p)=\frac{2-q_1}{2}\in[0,1)$. In the case of $p\in(2,\infty)$ we again use the duality to obtain the desired result.
\subsection{Estimates for short variations}\label{sec:shortvar}
In this section we focus on bounding the short variations, namely we want to establish the following estimate
\begin{equation}\label{shjumps:toshow}
    \norm[\Big]{\Big(\sum_{l=0}^\infty V^2\big(M_n f: n\in[2^l,2^{l+1}]\cap\mathbb{U}\big)^2\Big)^{1/2}}_{\ell^p(\ZZ^\Gamma)}\lesssim_{\calS_p} \|f\|_{\ell^p(\ZZ^\Gamma)},\quad f\in \ell^p(\ZZ^\Gamma).
\end{equation}
For this purpose, for $L\in\NN$, let us consider the following cut-off  short variations
\begin{equation*}
    \Big(\sum_{l=0}^LV^2\big(M_n f: n\in[2^l,2^{l+1}]\cap\mathbb{U}\big)^2\Big)^{1/2}.
\end{equation*}
Let $R_{p}(L)$ denote the smallest constant for which the following estimate holds
\begin{equation}\label{shjump:ad1}
    \Big\|\Big(\sum_{l=0}^LV^2\big(M_n f: n\in[2^l,2^{l+1}]\cap\mathbb{U}\big)^2\Big)^{1/2}\Big\|_{\ell^p(\ZZ^\Gamma)}\le R_{p}(L)\|f\|_{\ell^p(\ZZ^\Gamma)},\quad f\in \ell^p(\ZZ^\Gamma).
\end{equation}
By estimate \eqref{square} we know that $R_{p}(L)\lesssim_{L,\,p} 1$. Using again the bootstrap argument we will show that $R_{p}(L)\lesssim_p1$. The proof will proceed in a similar way as in the case of the long jumps hence we will omit some details. Without loss of generality we can assume that $R_{p}(L)>1$ and $L\in\NN$ is large. Let $\chi\in(0,1/10)$ and let $u\in\NN$ be a fixed large number. For each $n\in\NN$ we define the following function
\begin{equation}\label{shjump:IWproj}
    \Xi_n(\xi):=\sum_{a/q\in\Sigma_{\leq n^u}}\eta\big(2^{n(A-\chi I)}(\xi-a/q)\big),
\end{equation}
where $\eta$ is a bump function of the form \eqref{bumpfun}, $I$ is the $|\Gamma|\times|\Gamma|$ identity matrix, $A$ is a matrix defined in Section~\ref{sec:longjumps} and $\Sigma_{\leq n^u}$ is the set of the Ionescu--Wainger rational fractions related to the parameter $\varrho=(10u)^{-1}$. Recall that we may write $M_n f=T_{\ZZ^\Gamma}[m_n]f$. Next, we use functions \eqref{shjump:IWproj} to estimate the left hand side of \eqref{shjump:ad1} by
\begin{align}
 \Big\|\Big(\sum_{l=0}^LV^2\big(T_{\ZZ^\Gamma}[\Xi_l&m_n]f : n\in[2^l,2^{l+1}]\cap\UU\big)^2\Big)^{1/2}\Big\|_{\ell^p(\ZZ^\Gamma)}\label{shjump:var1}\\
&+\Big\|\Big(\sum_{l=0}^LV^2\big(T_{\ZZ^\Gamma}[(1-\Xi_l)m_n]f : n\in[2^l,2^{l+1}]\cap\UU\big)^2\Big)^{1/2}\Big\|_{\ell^p(\ZZ^\Gamma)}.
\end{align}
Similar to the case of the long jumps, the first expression  corresponds to the major arcs and the second one to the minor arcs in the Hardy--Littlewood circle method.
\subsubsection{Minor arcs}
Again, we start with the estimate for the minor arcs since it is relatively easy and follows the same rule as in the case of the long jumps. By the triangle inequality it is enough to show
\begin{equation*}
    \big\|V^2\big(T_{\ZZ^\Gamma}[(1-\Xi_l) m_n]f : n\in[2^l,2^{l+1}]\cap\UU\big)\big\|_{\ell^p(\ZZ^\Gamma)}\lesssim(l+1)^{-2}\|f\|_{\ell^p(\ZZ^\Gamma)}.
\end{equation*}
We note that for each $l\in\NN_0$ only a finite set of numbers from $[2^l,2^{l+1}]\cap\UU$ give a contribution to the above variational seminorm. Hence it is enough to prove that for some $a(l)\in\NN$ we have
\begin{equation}\label{shjump:ex1}
    \big\|V^2\big(T_{\ZZ^\Gamma}[(1-\Xi_l) m_{n/2^{a(l)}}]f : n\in[2^{l+a(l)},2^{l+1+a(l)}]\cap\NN_0\big)\big\|_{\ell^p(\ZZ^\Gamma)}\lesssim(l+1)^{-2}\|f\|_{\ell^p(\ZZ^\Gamma)}.
\end{equation}
We can prove the estimate above in the same way as \eqref{ex1}, for details see \cite[Section 5.1, pp. 689--691]{MST2}
\subsubsection{Major arcs}
Now, our aim is to estimate \eqref{shjump:var1}. In this case we will follow the approach presented in the proof of the estimate for the long jumps. The case of the short jumps is in some way easier since there is no need to consider small and large scales. In order to estimate \eqref{shjump:var1} we introduce new multipliers
\begin{equation*}
    \Xi_l^{j}(\xi)=\sum_{a/q\in\Sigma_{\leq l^u}}\eta\big(2^{lA+jI}(\xi-a/q)\big),\quad j\in\ZZ.
\end{equation*}
Then one may write (compare with Section~\ref{sec:SsDPL})
\begin{equation*}
    \Xi_l(\xi)=\sum_{-\floor{\chi l}\le j<l}\big(\Xi_l^{j}(\xi)-\Xi_l^{j+1}(\xi)\big)+\big(\Xi_l^{-\chi l}(\xi)-\Xi_l^{-\floor{\chi l}}(\xi)\big)+\Xi_l^{l}(\xi).
\end{equation*}
Next, we use the new multipliers and estimate \eqref{shjump:var1} by
\begin{align}
    &\Big\|\Big(\sum_{l=0}^LV^2\big(T_{\ZZ^\Gamma}\big[m_n\big(\sum_{-\floor{\chi l}\le j<l}\Xi_l^{j}-\Xi_l^{j+1}\big)\big]f : n\in[2^l,2^{l+1}]\cap\UU\big)^2\Big)^{1/2}\Big\|_{\ell^p(\ZZ^\Gamma)}\label{ad3}\\
   &+\Big\|\Big(\sum_{l=0}^LV^2\big(T_{\ZZ^\Gamma}\big[(m_n-m_{2^l})(\Xi_l^{-\chi l}-\Xi_l^{-\floor{\chi l}}+\Xi_l^{l})\big]f : n\in[2^l,2^{l+1}]\cap\UU\big)^2\Big)^{1/2}\Big\|_{\ell^p(\ZZ^\Gamma)}\label{ad4}.
\end{align}
One can prove that $\eqref{ad4}\lesssim \|f\|_{\ell^p(\ZZ^\Gamma)}$ by proceeding in a similar way as in \cite[pp. 692--694]{MST2}.\\
\indent Now let us go back to \eqref{ad3}. For $p\in(1,2]$ it is enough to show
\begin{equation}
\begin{aligned}
    \Big\|\Big(\sum_{l=0}^LV^2\big(T_{\ZZ^\Gamma}\big[m_n\big(\sum_{-\floor{\chi l}\le j<l}\Xi_l^{j}-\Xi_l^{j+1}\big)\big]f : n\in[2^l,2^{l+1}]\cap\UU\big)^2&\Big)^{1/2}\Big\|_{\ell^p(\ZZ^\Gamma)}\lesssim R_{p}(L)^{\frac{2-p}{2}}\|f\|_{\ell^p(\ZZ^\Gamma)}.\label{pinfs3}
\end{aligned}
\end{equation}
For $p\in(2,\infty)$ we will show that
\begin{equation}
    {\rm LHS}\eqref{pinfs3}\lesssim R_{p'}(L)^{(p-2)/2(p-1)}\|f\|_{L^p(\ZZ^\Gamma)}\label{pinfl4}
\end{equation}
where $1/p+1/p'=1$.
\subsubsection{Estimates for \eqref{pinfs3} and discrete Littlewood--Paley theory}
Now, we take a look at the left hand side of \eqref{pinfs3} in the case of $p\in(1,\infty)$. Let $\tilde{\eta}(x):=\eta(x/2)$ and define a new multiplier
\begin{equation*}
    \Delta_{l,\,s}^j(\xi)=\sum_{a/q\in\Sigma_{s^u}}\big[\eta\big(2^{lA+jI}(\xi-a/q)\big)-\eta\big(2^{lA+(j+1)I}(\xi-a/q)\big)\big]\Tilde{\eta}\big(2^{s(A-\chi I)}(\xi-a/q)\big),
\end{equation*}
where $\Sigma_{s^u}:=\Sigma_{\leq(s+1)^u}\setminus\Sigma_{\leq s^u}$ for $s\in\NN$ and $\Sigma_{0^u}:=\Sigma_{\leq 1}$. We see that
\begin{equation}\label{poprawki:eq2}
    \Xi_l^j(\xi)-\Xi_l^{j+1}(\xi)=\sum_{s=0}^{l-1}\Delta_{l,\,s}^j(\xi).
\end{equation}
Consequently, if we use \eqref{poprawki:eq2} and change the order of summation we see that the estimate \eqref{pinfs3} will follow if we prove that
\begin{equation}\label{shjump:ex6.1}
\begin{aligned}
    \Big\|\Big(\sum_{\substack{0\le l\le L,\\l\ge\max\{-j/\chi,\,j-1,s-1,i\}}}V^2\big(T_{\ZZ^\Gamma}\big[m_n\Delta_{l,\,s}^j\big]f : n\in[2^l,&2^{l+1}]\cap\UU\big)^2\Big)^{1/2}\Big\|_{\ell^p(\ZZ^\Gamma)}\\
    &\lesssim (s+1)^{-2}B_p(L)2^{-|j|\beta}\|f\|_{\ell^p(\ZZ^\Gamma)},
\end{aligned}
\end{equation}
for some $\beta=\beta_p>0$ where for $p\in(1,2]$ the constant $B_p(L)$ is equal to $R_p(L)^{(2-p)/2}$ and for $p\in(2,\infty)$ we have $B_p(L)=R_{p'}(L)^{(p-2)/2(p-1)}$. Now, if we apply the Rademacher--Menshov inequality for the short jumps \cite[Lemma 2.5]{MSZ2} we see it is enough to establish
\begin{equation}\label{shjump:sqfun}
\begin{aligned}
    \Big\|\Big(\sum_{\substack{0\le l\le L,\\l\ge\max\{-j/\chi,\,j-1,s-1,i\}}}\sum_{m=0}^{2^i-1}|T_{\ZZ^\Gamma}\big[(&m_{2^l+2^{l-i}(m+1)}-m_{2^l+2^{l-i}m})\Delta_{l,\,s}^j\big]f|^2\Big)^{1/2}\Big\|_{\ell^p(\ZZ^\Gamma)}\\
    &\lesssim(s+1)^{-2}(i+1)^{-2}B_p(L)2^{-|j|\beta}\|f\|_{\ell^p(\ZZ^\Gamma)}.
\end{aligned}
\end{equation}
\subsubsection{Estimates for square function in \eqref{shjump:sqfun}} For simplicity we denote $B_{l,m}=M_{2^l+2^{l-i}(m+1)}-M_{2^l+2^{l-i}m}$. At first we will prove \eqref{shjump:sqfun} in the case of $p=2$. If we approximate discrete multipliers by their continuous counterparts by using \cite[Proposition 4.18]{MSZ3} we can use estimates \eqref{Phi1} and \eqref{gaussum}, and \cite[Proposition 4.16]{MSZ3} to prove that on the support of $\Delta_{l,\,s}^j$ we have
\begin{equation}
    |m_{2^l+2^{l-i}(m+1)}-m_{2^l+2^{l-i}m}|\lesssim \min\big\{2^{-i}, (s+1)^{-u\delta}(2^{-|j|/|\Gamma|}+2^{-l/4})\big\}.
\end{equation}
Consequently, one has
\begin{equation*}
    |m_{2^l+2^{l-i}(m+1)}-m_{2^l+2^{l-i}m}|^2\lesssim 2^{-3i/2}(s+1)^{-\delta u/2}(2^{-|j|/(2|\Gamma|)}+2^{-l/8}).
\end{equation*}
Hence, by Parseval's inequality we get
\begin{equation}\label{shjump:m1}
\begin{aligned}
   \Big\|\big(\sum_{(l,m)\in\JJ}|B_{l,m}T_{\ZZ^\Gamma}[\Delta_{l,\,s}^j]f|^2\big)^{1/2}\Big\|_{\ell^{2}(\ZZ^\Gamma)}\lesssim 2^{-i/4}(s+1)^{-u\alpha/4}2^{-|j|\beta}\|f\|_{\ell^{2}(\ZZ^\Gamma)},
\end{aligned}
\end{equation}
with $\mathbb{J}:=\{(l,m)\in\ZZ^2:l\in[l_0,L],v\in[0,2^i-1]\}$ where $l_0=\max\{-j/\chi,\,j-1,\,s-1,i\}$.\\
\indent In the case $p\in(1,2]$ we will use discrete Littlewood--Paley theory (see Section~\ref{sec:SsDPL}) and the bootstrap Lemma~\ref{lemma28}. We will apply it with parameters $q_0=1$, $q_1=p$, $\vartheta=1/2$, to a countable set $\mathbb{J}$, the operator $B_{l,m}$ and the functions $g_{l,m}=T_{\ZZ^\Gamma}[\Delta_{l,\,s}^j]f$. It is easy to check that for every $q\in(1,\infty)$ we have
\begin{align*}
    \sup_{(l,m)\in\mathbb{J}}\|B_{l,m}\|_{L^1\to L^1}&\lesssim 2^{-i}\quad\text{and}\quad\|B_{\ast,\mathbb{J}}\|_{L^q\to L^q}\lesssim R_{q}(L).
\end{align*}
The first inequality follows from \cite[Proposition 4.16]{MSZ3} and the last inequality follows from the fact that for any $m=0,\ldots,2^i-1$ and any $l\in\NN$ one has the following pointwise estimate
\begin{equation*}
    |M_{2^l+2^{l-i}(m+1)}f-M_{2^l+2^{l-i}m}f|\lesssim V^2\big(M_n f:n\in[2^l,2^{l+1}]\cap\UU\big).
\end{equation*}
Therefore, by applying Lemma~\ref{lemma28} and inequality \eqref{eq:DPL1} we obtain
\begin{equation}\label{shjump:m}
\begin{aligned}
   \Big\|\big(\sum_{(l,m)\in\JJ}|B_{l,m}T_{\ZZ^\Gamma}[\Delta_{l,\,s}^j]f|^2\big)^{1/2}\Big\|_{\ell^{q_{1/2}}(\ZZ^\Gamma)}\lesssim R_{p}(L)^{1/2}\log(s+1)\|f\|_{\ell^{q_{1/2}}(\ZZ^\Gamma)}.
\end{aligned}
\end{equation}
Since $q_{1/2}<p\le2$, we can interpolate \eqref{shjump:m1} with \eqref{shjump:m} to get that \eqref{shjump:sqfun} holds in the case of $p\in(1,2]$. When $p\in(2,\infty)$ the desired result follows by duality since $B_{l,m}$ is a convolution operator.
 \subsection{Comments and remarks}
 Let us point out several remarks which can be made after reading our proof:
 \begin{enumerate}
    \item The proof of Theorem~\ref{Thm1:bootstrap} was reduced to proving the estimates for the long jumps and short variations related to the dyadic integers $\DD:=\{2^{n}:n\in\NN\}$. Here one could also apply a slightly different approach. Namely, for a fixed $\tau\in(0,1)$ we consider the subexponential dyadic integers $\DD_\tau:=\{2^{n^\tau}:n\in\NN_0\}$ and use Proposition~\ref{prop:longandshort} to get 
    \begin{equation*}
        \calS_p(M_t f:t>0)\lesssim_{\calS_p}\calS_p(M_{n}f:n\in\DD_\tau)+\norm[\Big]{\Big(\sum_{n=0}^\infty V^2\big(M_tf:t\in[2^{n^\tau},2^{(n+1)^\tau}]\big)^2\Big)^{1/2}}_{\ell^p(\ZZ^\Gamma)}.
    \end{equation*}
    Then the short variations can be estimated by using arguments from \cite{zk} (see also \cite{MSS,MSZ3}), which reduces matters to showing that
    \begin{equation*}
        \calS_p(M_{n}f:n\in\DD_\tau)\lesssim_{\calS_p}\|f\|_{\ell^p(\ZZ^\Gamma)}.
    \end{equation*}
    The inequality above can be proved by using techniques presented in the proof of estimates for the long jumps in Section~\ref{sec:longjumps}. There could occur some technical nuances in defining the Ionescu--Wainger projections \eqref{IWproj}, for details see \cite{MSS,MSZ3}. We decided to take the presented approach since we wanted to show that the bootstrap approach is possible in the context of the short variations. 
     \item It is not hard to notice that one could use our approach to prove the seminorm inequalities 
     \begin{equation*}
         \calS_p(M_{2^n}f:n\in\NN)\lesssim_{\calS_p} \|f\|_{\ell^p(\ZZ^\Gamma)},
     \end{equation*}
     where the seminorm $\calS_p$ is not of the form \eqref{supnorm}--\eqref{varnorm}. It is enough for $\calS_p$ to be subadditive, satisfy the square function estimate \eqref{square}, have some properties of splitting and to satisfy the Rademacher--Menshov inequality \eqref{eq:remark3}. Moreover, we need to know that the appropriate continuous counterpart holds. Besides that, it does not even have to control the maximal function like \eqref{domsup1}, \eqref{domsup2} or even \eqref{domweak}.
     In that case, in Lemma~\ref{lemma28} we can just control the norm of the maximal function $\|B_{\ast,\JJ}\|$ by the norm of the usual maximal function $\sup_{t>0}|M_t|$ since we know that the latter one is $\ell^p$-bounded.
     \item Our approach can be modified and applied to the truncated singular Radon operator,
      \begin{equation*}
     H_t f(x):=\sum_{y\in\Omega_t\cap\ZZ^k\setminus\{0\}}f(x-(y)^\Gamma)K(y),
        \end{equation*}
        where $K\colon\RR^{k}\setminus\{0\} \to \CC$ is a  Calder\'on--Zygmund kernel satisfying appropriate conditions. We can prove the inequality
        \begin{equation}\label{eq:remark4}
            \calS_p(H_tf:t>0)\lesssim_{\calS_p}\|f\|_{\ell^p(\ZZ^\Gamma)},
        \end{equation}
        where the seminorm $\calS_p$ is one of \eqref{oscillnorm}--\eqref{varnorm} (we excluded the maximal seminorm). It is known that the inequality \eqref{eq:remark4} holds for all seminorms \eqref{supnorm}--\eqref{varnorm}, see \cite{MST1,MST2,MSZ3,S}.  The main difference between $M_t$ and $H_t$ is that the operator $H_t$ is not positive, hence to use Lemma~\ref{lemma28} (i.e. to estimate $\|B_{\ast,\JJ}\|$) we need to know that
        \begin{equation*}
            \|\sup_{t>0}|H_t f|\|_{\ell^p(\ZZ^\Gamma)}\lesssim_p\|f\|_{\ell^p(\ZZ^\Gamma)}.
        \end{equation*}
        This is very consistent with the continuous case (see \cite[Theorem 2.28]{MSZ2}).
 \end{enumerate}

\end{document}